\documentclass{amsart}
\usepackage{graphicx, amssymb}
\usepackage[alphabetic,y2k,lite]{amsrefs}
\usepackage{tikz}
\usetikzlibrary{shapes.geometric}
\usetikzlibrary{calc}
\newcommand{\tikzsetup}{
  \tikzstyle tikzdefault=[semithick,>=stealth];
  \tikzstyle dotnode=[fill,circle,inner sep=1pt];
  \tikzstyle morphism=[circle,draw,inner sep=1pt, minimum
size=18pt, scale=0.8];
  \tikzstyle coupon=[draw,inner sep=1pt, minimum size=18pt,scale=0.8];
  \tikzstyle overline=[preaction={draw,line
width=2mm,white,-},semithick,>=stealth];
  \tikzstyle drinfeld center=[>=stealth,red]];
  \tikzstyle cell=[fill=black!10,draw, very thin];
   \tikzstyle midarrow=[fill, dart, dart tail angle=112,
    dart tip angle=53,inner sep=0.55pt,sloped,allow upside down];
   \tikzstyle revmidarrow=[fill, dart, dart tail angle=112,
    dart tip angle=53,inner sep=0.55pt,sloped,
    allow upside down,rotate=180];
}

\newcommand{\tzPairingIII}{
\begin{tikzpicture}[baseline=0pt]
\tikzsetup
\node[morphism] (ph) at (0,0) {$\ph$};
\node[morphism] (ph*) at (2,0) {$\ph^*$};
\node at (0,0.7) {$\dots$};
\node at (2,0.7) {$\dots$};
\draw[tikzdefault,->] (ph)-- +(120:1cm) node[pos=1.0,above,scale=0.8]
{$V_1$};
\draw[tikzdefault,->] (ph)-- +(60:1cm) node[pos=1.0,above,scale=0.8]
{$V_n$};
\draw[tikzdefault,<-] (ph*)-- +(120:1cm) node[pos=1.0,above,scale=0.8]
{$V_n$};
\draw[tikzdefault,<-] (ph*)-- +(60:1cm) node[pos=1.0,above,scale=0.8]
{$V_1$};
\draw[tikzdefault,->] (ph) -- (ph*) node[pos=0.5,below,scale=0.8] {$X_i$}; 
\end{tikzpicture}
}
\newcommand{\tzPairingIV}{
\begin{tikzpicture}[baseline=-1cm]
\tikzsetup
\node at (0.2,0) {$\dots$};
\node at (2.7,0) {$\dots$};
\draw[tikzdefault,<-] (0,0) arc(190:350:1.5cm);
 \node at(1.5,-1.5) {$V_1$};
\draw[tikzdefault,<-] (0.5,0.2) arc(190:350:1cm);
 \node at(1.5,-0.3) {$V_n$};
\end{tikzpicture}
}
\newcommand{\tzPairingV}{
\begin{tikzpicture}[baseline=0pt]
\tikzsetup
\node[circle, draw, dashed, minimum size=1.2cm,label=135:$A$] (A) at (0,0)
{} ;
\node[morphism] (ph) at (1.5,0) {$\ph$};
\node[morphism] (ph') at (2.5,0) {$\ph^*$};
\node[circle, draw, dashed, minimum size=1.2cm, label=45:$B$] (B) at (4,0)
{};
\draw[tikzdefault, ->, out=45,in=135] (A) to (ph) 
                node[pos=0.5,above] {$V_1$};
\draw[tikzdefault, ->,out=15,in=165] (A) to (ph);
\draw[tikzdefault, ->,out=-15,in=195] (A) to (ph);
\draw[tikzdefault, ->,out=-45,in=225] (A) to (ph) 
                node[pos=0.5,below] {$V_n$};
\draw[tikzdefault, ->,out=45,in=135] (ph') to (B) 
                node[pos=0.5,above] {$V_1$};
\draw[tikzdefault, ->,out=15,in=165] (ph') to (B);
\draw[tikzdefault, ->,out=-15,in=195] (ph') to (B);
\draw[tikzdefault, ->,out=-45,in=225] (ph') to (B) 
                node[pos=0.5,below] {$V_n$};
\end{tikzpicture}
}
\newcommand{\tzPairingVI}{
\begin{tikzpicture}[baseline=0pt]
\tikzsetup
\node[circle, draw, dashed, minimum size=1.2cm,label=135:$A$] (A) at (0,0)
{};
\node[circle, draw, dashed, minimum size=1.2cm,label=45:$B$] (B) at (3,0)
{};
\draw[tikzdefault, ->,out=45,in=135] (A) to (B) 
                node[pos=0.5,above] {$V_1$};
\draw[tikzdefault, ->,out=15,in=165] (A) to (B);
\draw[tikzdefault, ->,out=-15,in=195] (A) to (B);
\draw[tikzdefault, ->,out=-45,in=225] (A) to (B) 
                node[pos=0.5,below] {$V_n$};
\end{tikzpicture}
}

\newcommand{\tzGluingAxiomI}{
\begin{tikzpicture}[baseline=0pt]
\tikzsetup
\node[morphism] (ph) at (0,0) {$\ph$};
\node[morphism] (psi) at (2,0) {$\psi$};
\draw[tikzdefault] (ph) -- +(0,1cm) node[pos=0.8, left] {$A$};
\draw[tikzdefault] (psi) -- +(0,1cm) node[pos=0.8, right] {$B$};
\draw[drinfeld center] (ph)--(psi) node[pos=0.3,above,black]{$Z$};
\draw[->,overline](3,1)--(3,0) arc (0:-180:1cm) -- +(0,1)
node[pos=0.7,right] {$X$};
\end{tikzpicture}
}

\newtheorem*{theorem*}{Theorem}
\newtheorem{theorem}{Theorem}[section]
\newtheorem{lemma}[theorem]{Lemma}

\theoremstyle{definition}
\newtheorem{definition}[theorem]{Definition}

\newtheorem{example}[theorem]{Example}

\theoremstyle{remark}

\numberwithin{equation}{section}
\newcommand{\firef}[1]{Figure~{\rm\ref{#1}}}

\newcommand{\leref}[1]{Lemma~{\rm\ref{#1}}}

\newcommand{\seref}[1]{Section~{\rm\ref{#1}}}

\newcommand{\fig}[1]
{\raisebox{-0.5\height}%
{\includegraphics{#1}}}
\newcommand{\figscale}[2]
{\raisebox{-0.5\height}%
{\includegraphics[scale=#1]{#2}}}

\newcommand{\ccc}[1]{\underset{\scriptstyle #1}{\bullet}}

\newcommand{\<}{\langle}
\renewcommand{\>}{\rangle}
\newcommand{\xxto}{\xrightarrow}              

\newcommand{\one}{\mathbf{1}}
\newcommand{\kk}{\mathbf{k}}       
\newcommand{\N}{\mathcal{N}}       
\newcommand{\DD}{\mathcal{D}}      
\newcommand{\C}{\mathcal{C}}      
\newcommand{\M}{\mathcal{M}}      
\newcommand{\Vect}{\mathcal{V}ec}  

\newcommand{\Ga}{\Gamma}

\newcommand{\ph}{\varphi}

\DeclareMathOperator{\Irr}{Irr}

\DeclareMathOperator{\Hom}{Hom}

\DeclareMathOperator{\ev}{ev} 
\DeclareMathOperator{\coev}{coev} 
\DeclareMathOperator{\Obj}{Obj}
\DeclareMathOperator{\Dim}{Dim}
\begin{document}
\title{Turaev-Viro invariants as an extended TQFT II}

\author{Benjamin Balsam}
   \address{Department of Mathematics, SUNY at Stony Brook, 
            Stony Brook, NY 11794, USA}
    \email{balsam@math.sunysb.edu}
    \urladdr{http://www.math.sunysb.edu/\textasciitilde balsam/}
    \thanks{This  work was partially suported by NSF grant DMS-0700589 }

\begin{abstract}
In this paper, we present the next step in the proof  that $Z_{TV,\C} = Z_{RT, Z(\C)}$, namely that the theories give the same 3-manifold invariants. In future papers we will show that this equality extends to an equivalence of TQFTs. This paper is a continution of \ocite{mine}.

\end{abstract}
\maketitle
\section*{Introduction}

In \ocite{mine}, we defined a version of the Turaev-Viro TQFT for 3-manifolds with corners. The theory coincides with the classical TV for ordinary 3-manifolds with boundary, and  thus takes as input data a spherical category $\mathcal{C}$. Instead of considering manifolds with corners, we consider ordinary 3-manifolds with boundary, replacing the corners with framed embedded tubes. As in \ocite{BK}, these are the same extended 3-manifolds used in the Reshetikhin-Turaev (RT) theory.

We also made use of a category associated to $\mathcal{C}$, its Drinfeld Center $Z(\mathcal{C})$. If $\mathcal{C}$ is spherical, it can be shown that  $Z(\mathcal{C})$ is modular, and, in particular, is braided.  The extended theory assigns the category $Z(\mathcal{C})$ to a circle, and we computed that for the n-punctured sphere with boundary components labeled $Y_{1},...Y_{n}$,  $Z_{TV,\mathcal{C}}(S^2_{n},Y_{1},...,Y_{n}) = \Hom_{Z(\mathcal{C})}(\mathbf{1},Y_{1},...,Y_{n})$, the same space that $Z_{RT,Z(\mathcal{C})}$ produces.

In this paper, we finish proving that for a closed 3-manifold $\M$ (possibly with an embedded link inside), $Z_{TV,\mathcal{C}}(\M) = Z_{RT,Z(\mathcal{C})}(\M)$. Turaev and Virelizier recently posted a proof of this formula \ocite{similar}, but the methods are different, and in particular involves state sums on skeletons of 3-manifolds and the theory of Hopf monads in monoidal category. In contrast, this paper employs the methods from \ocite{mine}, in which the Turaev-Viro TQFT is described as a 3-2-1 extended theory in the sense of \ocite{lurie}. This extended theory is a generalization of the classical results of \ocite{TV} and \ocite{barrett}.

The organization of this paper is as follows. In section 1, we recall some definitions and results that will be used throughout the paper. In section 2, we prove that the theories coincide for $S^3$ with a link inside. We do this by decomposing $S^3$ into a finite collection of \textit{building blocks}, demonstrating  that the theories coincide on these blocks, and then using the gluing axiom.  In section 3, we describe graphically the vector space assigned to the torus and examine the action of the mapping class group. In particular, we show that the generators $T$ and $S$ of the mapping class group act by multiplication by the twist and s-matrices respectively, just as they do in $Z_{RT}$. Finally, in the last section, we prove the surgery formula for our theory, which implies the main theorem as a corollary. The appendix contains some of the more detailed computations.

\subsection*{Acknowledgments}
The author would like to thank Sasha Kirillov for his extensive help and guidance in writing this paper.

\section{Preliminaries}
In this section, we recall without proof some important results that will be useful in this paper. All proofs may be found in \ocite{mine}, which is a prerequisite to reading this paper. \\

Throughout the paper $\mathcal{C}$ will denote a spherical fusion category over some algebraically closed field of characteristic 0. Recall that this means that $\mathcal{C}$ is semisimple, with finitely many isomorphism classes of simple objects. We will also assume that $\mathcal{C}$ is strict pivotal and that the unit object $\one \in \mathcal{C}$ is simple.
For $X \in \mathcal{C}$, we denote by
$$
d_{X} = \dim X\ = tr(Id_{X}) \in \kk
$$
For each simple object $X_i \in \mathcal{C}$, we fix a choice of square root $\sqrt{d_{X_i}}$ so that that for $\one \in \mathcal{C}$, $\sqrt{d_{\one}} = 1$, and for any simple object $X$, $\sqrt{d_{X}} = \sqrt{d_{X^*}}$.
We also fix an element
\begin{equation}
\DD = \sqrt{\displaystyle \sum_{x \in Irr(\mathcal{C})} d^2_{X}} 
\end{equation}
which we call the dimension of $\C$ or Dim($\mathcal{C}$).
By results of \ocite{ENO}, $\DD \neq 0$ and for simple $X$, $d_{X} \neq 0$. 
For a category $\mathcal{C}$, we define a functor $\C^{\boxtimes n} \longrightarrow \Vect$ by
\begin{equation}\label{e:vev}
 \<V_1,\dots,V_n\>_{\mathcal{C}}=\Hom_\C(\one, V_1\otimes\dots\otimes V_n\>
\end{equation}
where $V_1, \dots, V_n \in \C$. The subscript $\C$ is included in \ref{e:vev} to remind the reader which category we are operating in. We will omit it when there is no potential ambiguity.
For any objects $A,B \in \Obj\mathcal{C}$, we have a non-degenerate pairing $\Hom_{\C}(A,B)\otimes \Hom_{\C}(A^*,B^*)\rightarrow \kk$ given by
\begin{equation}\label{e:pairing}
(\ph, \ph')=(\one\xxto{\coev_A}A\otimes A^*\xxto{\ph\otimes \ph'}
  B\otimes B^*\xxto{\ev_B}\one)
\end{equation}
\begin{lemma}\label{l:composition2}
Let $X$ be a simple object. Define the composition map
\begin{equation}\label{e:composition2}
\begin{aligned}
 \<V_1,\dots,V_n, X\>\otimes\<X^*, W_1,\dots,
W_m\>&\to\<V_1,\dots,V_n, W_1,\dots, W_m\>\\
\ph\otimes\psi&\mapsto \ph\ccc{X}\psi= \sqrt{d_X}\ \ev_X\circ
(\ph\otimes\psi)
\end{aligned}
\end{equation}
Then the composition map agrees with the pairing: 
$$(\ph\ccc{X} \psi,\psi'\ccc{X^*}\ph')= (\ph,\ph')(\psi',\psi)
$$  
\end{lemma}
In addition to the category $\C$, our extended TV theory makes use of a related category, the \textit{Drinfeld Center} of $\C$.
Let $\C$ be a spherical fusion category. The \textit{Drinfeld Center} of $\C$, denoted $Z(\C)$ is the category with
\begin{itemize}
\item Objects are pairs $(Z,\varphi_{Z})$, where $V \in \Obj\C$, and $\varphi_Z: Z\otimes X \rightarrow X \otimes Z$ a natural isomorphism for all $X \in \C$. The collection of morphisms $\{\varphi_Z\}_{Z \in Z(\C)}$ is aptly named a half-braiding, and these morphisms must satisfy certain coherence conditions.
\item Morphisms $\Psi: (Z,\varphi_{Z}) \longrightarrow (Y, \varphi_{Y})$ in $Z(\C)$ are morphisms $\Psi: Z \rightarrow W$ in $\C$ that are compatible with the half-braiding.
\end{itemize}
For a more precise definition see \ocite{muger2}.
The following important theorem is due to Mueger.
\begin{theorem}
 $Z(\C)$ is a modular category; in particular, it is
    semisimple with finitely many simple objects, it is braided and has
    a pivotal structure which coincides with the pivotal structure on
    $\C$. 
\end{theorem}

\begin{figure}[ht]
\figscale{.5}{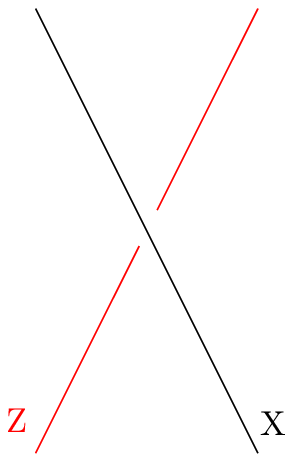}
\caption{The half-braiding $\varphi_Z: Z\otimes X \rightarrow X\otimes Z$}
\end{figure}
We will make heavy use of the graphical techniques developed in, e.g. \ocite{turaev}  to represent morphisms in the category $\C$. These techniques will greatly simplify the calculations that follow. As in \ocite{mine}, morphisms in the category of tangles are read bottom to top. We also allow for circular coupons in addition to the standard rectangular ones. Tangle strands are to be labelled by objects of $\C$ and coupons, with morphisms in the appropriate $Hom$ spaces (which are vector spaces in a spherical category). Note that $\C$ is not equipped with a braiding, so we will not allow edges to cross. When evaluating such morphisms, we use the following conventions:
\begin{enumerate}
\item If a figure contains a pair of circular coupons, one with outgoing edges labelled $V_1,\dots V_n$ and the other with edges labelled $V^*_{n}, \dots, V^*_{1}$ and the coupons are labelled by a pair of letters, such as $\varphi$ and $\varphi^*$, it stands for summation over dual bases with respect to \ref{e:pairing}. Alternatively, for sake of brevity we will on occasion represent such pairs of circular coupons by pairs of vertices of the same color.
\item For a spherical fusion category $\C$, we refer to the set of isomorphism classes of simple objects as Irr($\C$). By an abuse of notation, we use the term \textit{simple object} to denote an element of this set
\item If a diagram contains an unlabelled edge, we sum over all possible labellings of that edge with simple objects $X_i \in \C$, each with weight $d_i$.
\item When labelling edges with simple object $X_i$ we will often just use the label $i$. 
\item We will sometimes neglect to orient edges in diagrams when doing so would prove cumbersome. The orientations are only important when, for example, using the composition map or pairing dual vertices, and we will be careful in these cases.
\end{enumerate}
The following lemma will be very useful in forthcoming computations.
\begin{lemma}\label{l:pairing}
\par\noindent
\begin{enumerate}
 \item 
 If $X$ is simple and $\ph\in \<X,A\>$, $\ph'\in\<A^*,X^*\>$ then 
$$
\fig{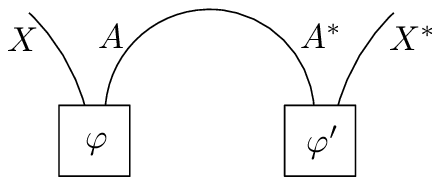}=\frac{(\ph,\ph')}{d_X}\quad \fig{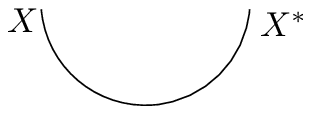}
$$
\item 
 $$
  \sum_{i\in\Irr(\C)} d_i 
  \tzPairingIII=\tzPairingIV
 $$

\item If the subgraphs $A$, $B$ are not connected, then 
 $$
 \tzPairingV=\tzPairingVI
 $$
\end{enumerate} 
\end{lemma}
We also have the following isomorphism which can be realized graphically.
\begin{lemma}\label{l:gluing_isom2} 
For any $A,B\in \Obj \C$, the map
\begin{equation}\label{e:gluing_isom2}
\begin{aligned}
\bigoplus_{Z\in \Irr(Z(\C))} \<Z, A\>\otimes \<Z^*, B\>&\to
  \bigoplus_{X\in \Irr(\C)}\<A, X , B, X^*\>\\
\ph\otimes \psi&\mapsto\bigoplus_{X\in \Irr(\C)}
\frac{\sqrt{d_X}\sqrt{d_Z}}{\DD} \quad \tzGluingAxiomI
\end{aligned}
\end{equation}
is an isomorphism.
\end{lemma}
As shown in [RT], we one can \textit{evaluate} a graph ($\Gamma; X_i)$ with edges colored by irreducible objects of $\C$ and with vertices (or "coupons") labeled by appropriate morphisms  to obtain a number $F((\Gamma; X_i)) \in \kk$. The axioms of a spherical category imply that we may actually view the graph as embedded in the sphere; any two \textit{flattenings} to a planar graph will evaluate to the same number.  In \ocite{mine}  the authors considered a generalization of these graphs.
\begin{definition}
An \textit{extended graph} $\hat{\Gamma}$ consists of a usual graph $\Gamma$ embedded in $S^2$ along with a finite collection of strands \{$\gamma$\}  that terminate on vertices of $\Gamma$. These strands are allowed to intersect the edges of $\Gamma$ away from vertices, and may  cross over or under one another.
\end{definition}
In this paper, the strands $\{\gamma\}$ will be colored red.
A labelling of an extended graph $\hat{\Gamma}$ is function $\eta$ which assigns to each oriented edge of $\Gamma$ an object of $\C$, to each red strand, an object of $Z(\C)$, and to each vertex or coupon, a morphism in the appropriate Hom space (all Homs in $\C$!). Given a labelled extended graph $(\hat{\Gamma}, \eta)$ we may flatten the graph by removing a point in $S^2$ and evaluate the graph as in [RT], where we replace all crossings with the half-braiding. 
\begin{theorem}
 The number $Z_{RT}(\hat\Ga)\in \kk$ does not depend on the choice of a
  point to remove from $S^2$ and thus defines an
  invariant of colored extended graphs on the sphere. Moreover, this number is
  invariant under homotopy of strands \{$\gamma$\}.
\end{theorem}
We have a natural inclusion $\Obj(Z(\C)) \subset \Obj(\C)$ and for $Y, Z \in Z(\C)$, $Hom_{Z(\C)}(Y,Z) \subset Hom_{\C}(Y,Z)$.
\begin{lemma}\label{l:projector}
  Let $Y,Z\in \Obj Z(\C)$. Define the operator 
  
  $P\colon \Hom_\C(Y,Z)\to
  \Hom_\C(Y,Z)$ by the following formula:
  $$
    P\psi=\frac{1}{\DD^2}\sum_{X\in \Irr(\C)} d_X\quad
    \fig{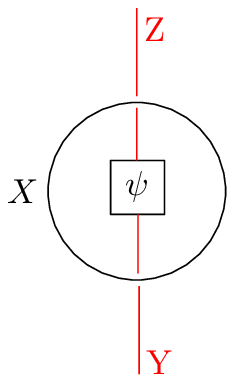}
  $$
  \begin{enumerate}
  \item $P$ is a projector onto the subspace   
  $\Hom_{Z(\C)}(Y,Z)\subset   \Hom_\C(Y,Z)$. 
  \item If $Y, Z$ are simple objects, then $P\psi = \delta_{Y,Z}\frac{1}{d_Z}tr(\psi)Id_{Y}$
  \end{enumerate}
\end{lemma}

We close this section by summarizing the construction of TV as an extended TQFT. We assume the reader is familiar with the standard TV theory.

\begin{itemize}
\item  To each closed 1-manifold, we should assign an abelian category. We take $Z_{TV,\C}(\displaystyle \sqcup_{i=1}^{n} S^1) = Z(\C)^{\boxtimes n}$, $Z(\emptyset) = \Vect$.
\item To a cobordism between 1-manifolds, we should assign a functor between corresponding categories. In particular, if $\N = S^2$  with n punctures, we may view it as a cobordism $\N: \partial N \to \emptyset $. Then
\begin{center}
 $Z_{TV, \C}(\N): Z(\C)^{\boxtimes n} \to \Vect$ is the functor $Hom_{Z(\C)}(\one, -\otimes - \otimes \dots \otimes -)$.
\end{center}
If the boundary components of $\N$ are colored by objects $Z_1 \dots Z_n \in Z(\C)$, then $Z_{TV, \C}(\N) = Hom(\one, Z_{1} \otimes \dots \otimes Z_{n})$, a vector space.
\item To a cobordism $\Phi: \N_1 \to \N_2$ between manifolds with boundary, we assign a natural transformation between associated functors. Equivalently, if we color $\N_1, \N_2$ as above, we obtain a linear map $Z(\Phi): Z(\N_1) \to Z(\N_2)$.
\end{itemize}
Instead of surfaces with boundary, and 3-manifolds with corners, we adopt an equivalent formalism, replacing corners with embedded disks and tubes. 

We consider decompositions of manifolds that are more general than triangulations, but still less general than arbitrary cell decompositions. See \ocite{mine} for details. For the purposes of this paper, it suffices to note that these so-called \textit{combinatorial} decompositions are well-behaved: any two decompositions are related by a finite sequence of moves analogous to the Pachner moves for triangulations, and the resulting invariants are independent of the chosen decomposition. The one important detail is that combinatorial structures equip each embedded tube of a 3-manifold with a longitude, which specifies the framing. These longitudes will be red in color in accordance with the convention described earlier, since they will be labelled by elements of $Z(\C)$.
For a detailed description of the state-sum construction, see \ocite{mine}. This construction is a direct generalization of that of \ocite{barrett}, and behaves much in the same way. In particular, we  have gluing axioms for surfaces as well as 3-manifolds (See Theorems 8.4, 8.5 in \ocite{mine}).

\section{Sphere }\label{s:s3link}
In this section we compute the TV state sum for several extended 3-manifolds, which we call \textit{generators}. If $S^3_L$ denotes $S^3$ with an embedded link $L$ inside, we can decompose $S^3_L$ into a finite union of these generators. Using the computations in this section and the gluing axiom for our TQFT, we conclude that the theories give the same answer for $S^3_L$. In what follows, all links are framed and oriented.

Consider the following extended 3-manifold structure on $\N$ where $\N$ is the cobordism  between 3-punctured spheres that interchanges two of the embedded disks with longitudes labeled as pictured \footnote{We have removed a solid cylinder from the figure to make the diagram more manageable}. 
 Clearly, the two picture are homeomorphic
\begin{figure}[ht]
\begin{center}
$\N$ =\hspace{.5cm} \figscale{.5}{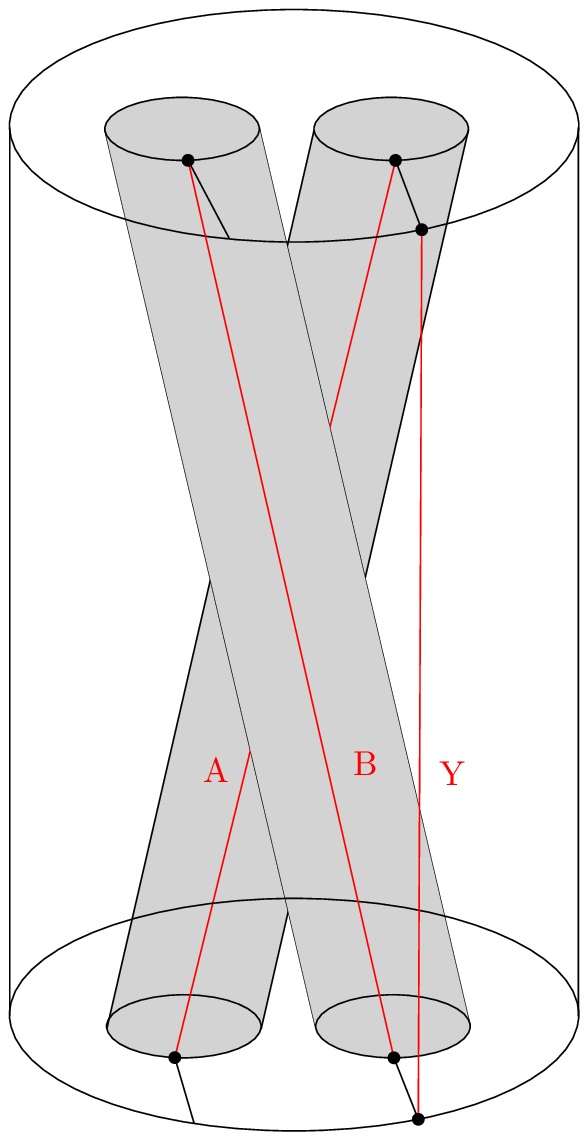} \hspace{.5cm}$\cong$ \hspace{.5cm} \figscale{.5}{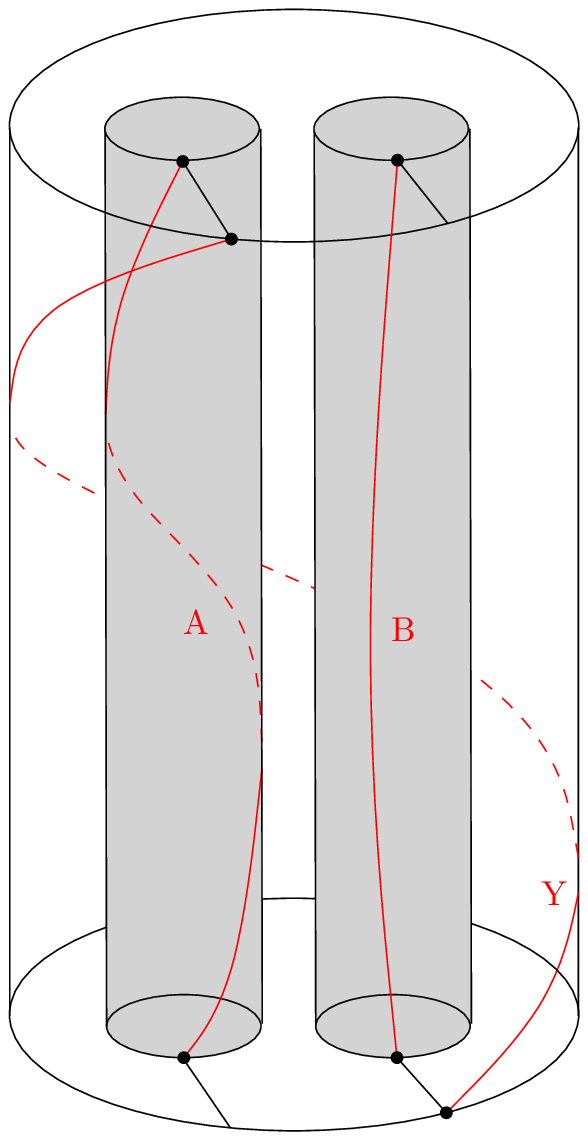}
\end{center}
\end{figure}
If we take a modular category as input data, Reshetikhin-Turaev theory gives  $Z_{RT}(\mathcal{N}) = Id_Y \otimes \sigma_{AB}$, where $\sigma$ is the braiding \ocite{BK}. We now show that Turaev-Viro theory gives the same answer.
\begin{lemma} \label{l:braid}
Let $\C$ be a spherical category. Then there is a canonical isomorphism $Z_{TV,\C}(\N) \cong Z_{RT, Z(\C)}(\N)$
\end{lemma}
\begin{proof}
See \seref{s:appendix}
\end{proof}
Now let $\mathcal{M} = S^2 \times I$ with a single open embedded tube colored by $Y \in Z(\C)$ as shown. As in the previous lemma, we have removed a solid cylinder from $\mathcal{M}$. By definition, $Z_{RT, Z(\C)}(\M) = ev_Y: Y^* \otimes Y \to \one$, the evaluation map.
\begin{figure}[ht]
\figscale{.4}{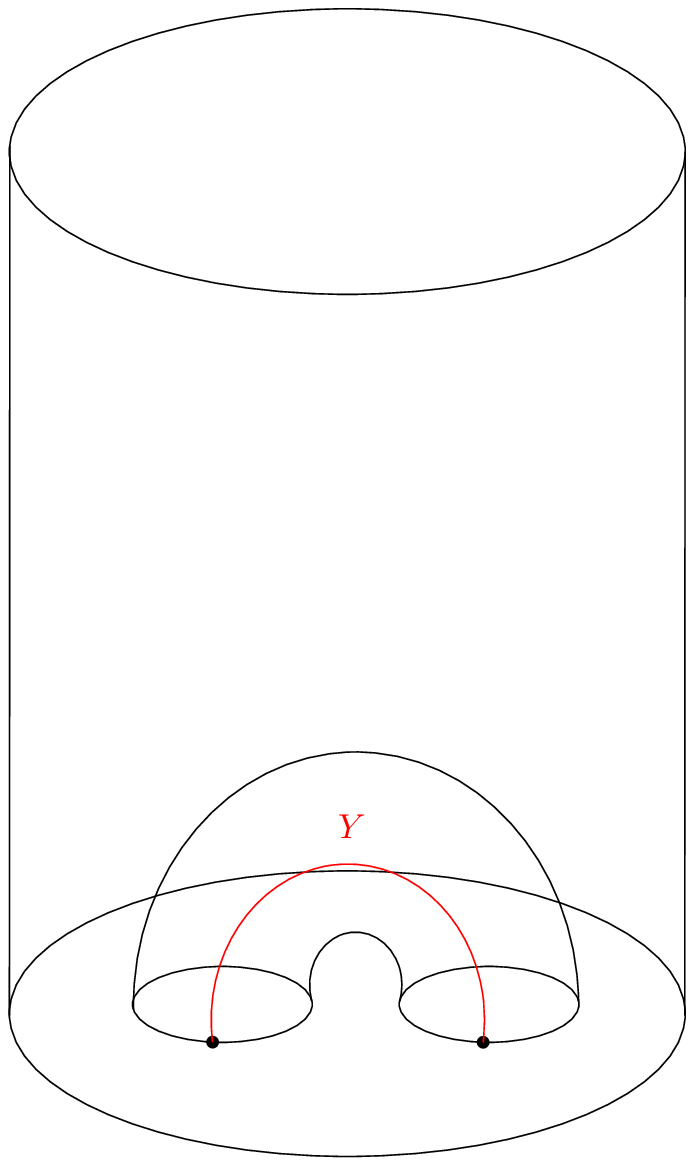}
\end{figure}
\begin{lemma} \label{l:cap}
$Z_{TV,\mathcal{C}}(\mathcal{M}) \cong Z_{RT,Z(\mathcal{C})}(\mathcal{M})$. 
\end{lemma}
\begin{proof}
See \seref{s:appendix}
\end{proof}

It follows by an identical calculation that $Z_{TV}(\mathcal{M'})$ gives the coevaluation map, where $\mathcal{M'}$ is similar to $\mathcal{M}$ but inverted. \\
We can now use the above two lemmas to state the following result.
\begin{theorem}\label{t:sameforsphere}
Let $M = S_{L}^3$ be the 3-sphere with an embedded link $L$ inside with components colored by $Y \in Irr(Z(\C))$. Then $Z_{TV,\mathcal{C}}(M) = Z_{RT,Z(\mathcal{C})}(M)$.
\end{theorem}
\begin{proof}
After isotoping $L$ appropriately, we can cut $M$ into regions, each of which is isomorphic to $\mathcal{N}, \mathcal{M}$ or $\mathcal{M'}$ from the above lemmas or to $B^3$. The theorem then follows from the lemmas and the gluing axiom \ocite{mine}. 
\end{proof}
Note that it is known that for RT, $Z_{RT, \mathcal{A}}(S^3_{L}) = \frac{1}{\Dim(\mathcal{A}}F(L)$, so the theorem gives 
\begin{equation}
Z_{TV, \mathcal{C}}(S^3_{L}) = Z_{RT, Z(\mathcal{C})}(S^3_L) = \frac{1}{\Dim(Z(\C))}F(L) = \frac{1}{\mathcal{D}^2}F(L)
\end{equation}
which can also be verified by direct computation. Here we use the fact that Dim(Z($\mathcal{C}$))= Dim($\mathcal{C})^2$ \ocite{muger2}.

\section{Computations with the Torus}\label{s:torus}
In the previous section, we established that $Z_{RT}(S^{3}_{L}) = Z_{TV}(S^{3}_{L})$, where $S^{3}_{L}$ is the 3-sphere with an embedded link $L$ inside. It is a classical result that any connected, closed 3-manifold may be obtained from $S^3$ via surgery along a framed link $L$, or, more precisely, along a tubular neighborhood of $L$. A tubular neighborhood of a link is simply a disjoint union of solid tori. In this section, we study the vector space $Z_{TV}(\mathbb{T}^2)$ and describe graphically an inner product, an orthonormal basis and action of the mapping class group (MCG) of the torus. We denote the standard 2-torus $S^1 \times S^1$ by $\mathbb{T}^2$.

The following result is well-known (see e.g. \ocite{muger2}):
\begin{lemma} \label{l:torusspace}
$Z_{TV,\mathcal{C}}(\mathbb{T}^2)$ has basis indexed by isomorphism classes of irreducible objects of $\mathcal{Z(C)}$. 
\end{lemma}
\begin{proof}
The torus may be obtained from the 2-punctured sphere, $S_{2}^{2}$ by gluing together its two boundary circles. It follows directly from the gluing axiom for surfaces (Theorems 8.4, 8.5 in \ocite{mine}) that 
\begin{align} \label{l:ztorus}
Z_{TV,\mathcal{C}}(\mathbb{T}^2) = \displaystyle\bigoplus_{Z \in Irr(Z(\mathcal{C}))}\<Z,Z^{*}\>_{Z(\C)}.
\end{align}
\end{proof}
We can also see this explicitly by computing $Z_{TV}(\mathbb{T}^2 \times I)$, as is done in \seref{s:appendix}.

We denote by $\textbf{T}_{Z}^2$ the solid torus with a closed embedded tube inside with (untwisted) longitude labeled by $Z \in Irr(Z(\mathcal{C}))$ as shown in \firef{f:solid_torus}. Then $Z_{TV,\mathcal{C}}(\textbf{T}_{Z}^2)$ is a vector in $Z_{TV,\mathcal{C}}(\mathbb{T}^2)$. We denote this vector by $[Z]$.
\begin{figure}[ht]
\figscale{.5}{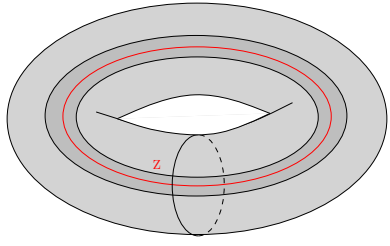}
\caption{The solid torus with a closed embedded tube labelled by $Z \in Z(\mathcal{C})$ }\label{f:solid_torus}
\end{figure}

\begin{lemma}
\{$[Z]\}_{Z \in Irr(Z(\mathcal{C}))}$ form a basis in $Z(\mathbb{T}^2)$.
\end{lemma}
A simple computation shows that $[Z] = \frac{1}{\sqrt{d_Z}} coev_Z$, under the identification \ref{l:ztorus} . Note that $coev_{\C}|_{Z(\C)} = coev_{Z(\C)}$. 
\section{Surgery}\label {s:surgery}

If $M$, $N$ are manifolds with boundary and $\varphi: \partial M \longrightarrow \partial \overline{N}$ is a homeomorphism, we may glue $M$ and $N$  along their boundaries to obtain a closed 3-manifold, denoted $M \sqcup_{\varphi} N$. Such a map $\varphi$ induces a linear map between the corresponding vector spaces, and we denote this map $\varphi_{*}$: $Z(\partial M) \to Z(\partial \overline{N})$. By the gluing axiom, $Z(M \sqcup_{\varphi} N) = (\varphi_{*}Z(M),Z(N))$. 
\begin{lemma}
$M \sqcup_{\varphi} N$ only depends on the isotopy class of $\varphi$
\end{lemma}
For a proof of this lemma, see \ocite{BK}.
\begin{definition}
For a closed manifold $M$, the \textit{mapping class group} $\Gamma(M)$ is the group of isotopy classes of homeomorphisms $M \longrightarrow M$.
\end{definition}
Since any 3-manifold may be obtained from $S^3$ by surgery along an embedded link, or equivalently along a collection of solid tori, we will be primarily concerned with the following example:
\begin{example}
$\Gamma(\mathbb{T}^2) = SL_{2}(\mathbb{Z})$, which is generated by
$S = \begin{pmatrix} 0&-1\\ 1&0 \end{pmatrix}$ and $T = \begin{pmatrix} 1&1\\ 0&1 \end{pmatrix}$. If we pick generators $\alpha, \beta$ for $H_{1}(\partial \textbf{T}^2,\mathbb{Z})$ as in \firef{f:torus_twist}, then $S$ acts by interchanging \footnote{More precisely S($\alpha$) = $\beta$, S($\beta$) = -$\alpha$} $\alpha$ and $\beta$  and $T$ is a Dehn twist (See \firef{f:torus_twist}). Further, the action of $T$ extends to a homeomorphism of the solid torus.
\begin{figure}[ht] 
$T:$ \hspace{.5cm} $\figscale{.4}{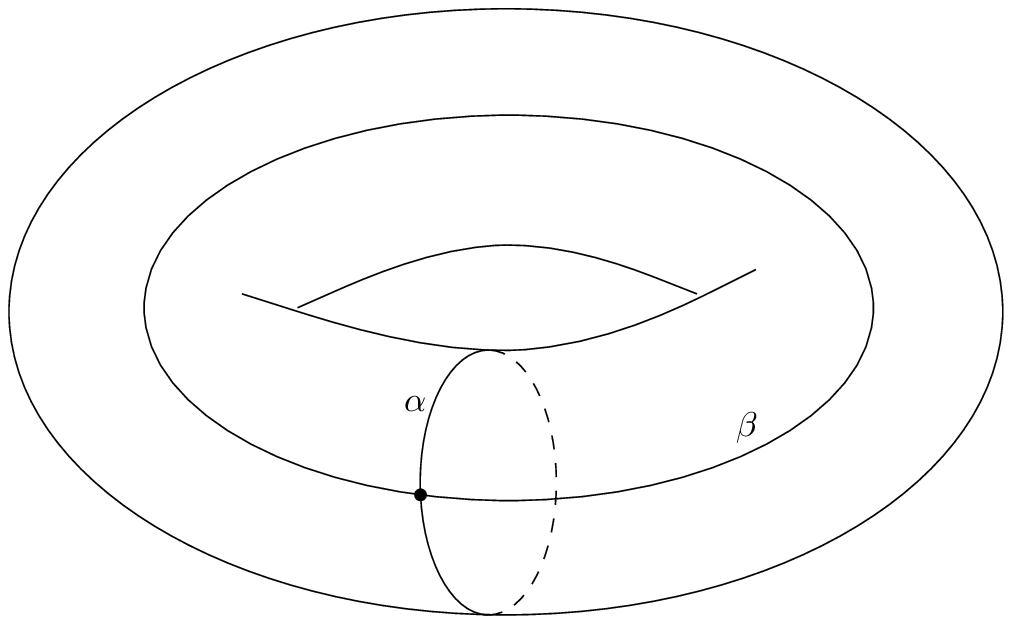} \hspace{.5cm} \longmapsto \hspace{.5cm} \figscale{.4}{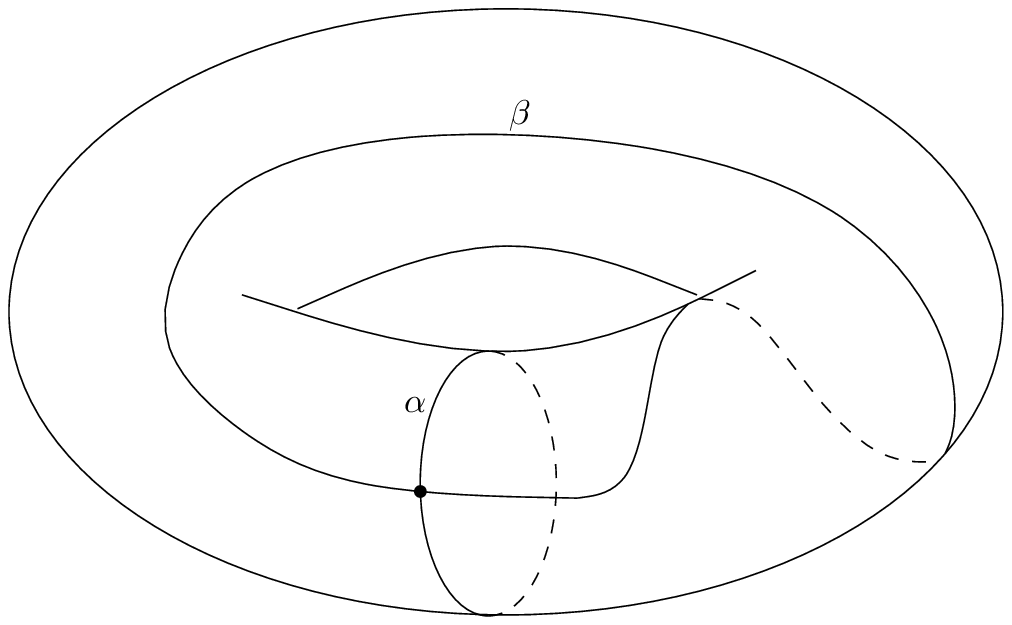}$
\caption{The action of $T$ on the Torus}\label{f:torus_twist}
\end{figure}
\end{example}
We now describe an inner product on $Z(\mathbb{T}^2)$.
\begin{lemma}\label{l:innerproduct}
Define a bilinear form on $Z(\mathbb{T}^2)$  by $\<[Z],[W]\> = Z(\textbf{T}^2_{Z} \sqcup_{U} \textbf{T}^2_{W})$, where $U =\begin{pmatrix} -1&0\\ 0&1 \end{pmatrix}$. Then $\{[Z]\}_{Z \in Irr(Z(\mathcal{C}))}$ is orthonormal with respect to this form.
\end{lemma}
\begin{proof}
See \seref{s:appendix}.
\end{proof}

Since $Z(\mathcal{C})$ is a modular category, it has an s-matrix $\tilde{s}$ (\ocite{BK}) and a twist matrix $t$ where $t_{i,j}$ is zero for $i \neq j$ and is defined for $i =j$ by
\begin{align*}
 \figscale{.5}{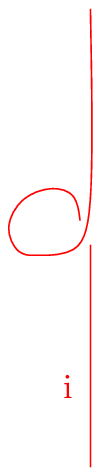} = t_{i,i}\figscale{.5}{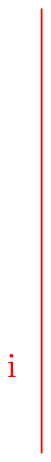} 
\end{align*}
Let $S$ and $T$ be as above. The following theorem shows that $S$ and $T$ act on $Z(\mathbb{T}^2)$ by the s-matrix and the twist matrix respectively.
\begin{theorem} \label{t:computation}
\begin{enumerate}
\item $S_{*}[Z] = \displaystyle \sum_{W \in Irr(Z(\mathcal{C}))} \frac{\tilde{s}_{ZW}}{\DD^2}[W]$. 
\item $T_{*}[Z] = t_{Z,Z}[Z]$
\end{enumerate}
\end{theorem}
\begin{proof}
\begin{enumerate}
\item By the gluing axiom, $(S_{*}[Z], [W]) = Z_{TV, \mathcal{C}}(\textbf{T}^2_{Z} \sqcup_{S} \textbf{T}^2_{W})$ = $Z_{TV,\mathcal{C}}(S^3_{L})$, where L is the Hopf link with components labelled by $Z$ and $W$. By \ref{t:sameforsphere}, this equals $Z_{RT, Z(\mathcal{C})}(S^3_L) = \tilde{s}_{ZW}$.
\item By the gluing axiom  $(T_{*}[Z], [W]) = Z_{TV, \mathcal{C}}(\textbf{T}^2_{Z} \sqcup_{T} \textbf{T}^2_{W})$. This manifold is homeomorphic to $S^2 \times S^1$ with two unlinked embedded closed tubes labelled $Z$ and $W$ respectively, the $Z$ tube with a single positive twist. This follows directly from the fact that $\textbf{T}^2 \sqcup_{U} \textbf{T}^2 = S^2 \times S^1$ combined with the fact that the map $T$ extends to a homeomorphism of solid tori. Replacing the twist with a multiplicative factor $t_{Z,Z}$, we get $Z_{TV, \mathcal{C}}(\textbf{T}^2_{Z} \sqcup_{T} \textbf{T}^2_{W}) = t_{Z,Z}Z_{TV, \mathcal{C}}(\textbf{T}^2_{Z} \sqcup_{U} \textbf{T}^2_{W}) = t_{Z,Z}\delta_{Z,W}$, where the final equality holds by \leref{l:innerproduct}.
\end{enumerate}
\end{proof}
The following computations will be particularly useful:
\begin{enumerate}
\item  $S_{*}[\one] = \displaystyle \sum_{W \in Irr(Z(\mathcal{C}))} \frac{d_{W}}{\DD^2}[W]$. This follows immediately from Theorem 4.5(1) and the fact that $\tilde{s}_{\one,Z} = d_{Z}$.
\item $S_{*}\displaystyle \sum_{Z} \frac{d_{Z}}{\DD^2}[Z] = [\one]$. This follows from the equation $\displaystyle \sum_{Z} d_{Z} \tilde{s}_{Z,W} = \delta_{W, \one}\DD^2$ (See \ocite{BK}, Chapter 3).

\end{enumerate}
Let us briefly outline what we have accomplished so far. 
\begin{enumerate}
\item In \seref{s:s3link}, we demonstrated that $Z_{TV,\mathcal{C}}(S^3_{L}) = Z_{RT, Z(\mathcal{C})}(S^3_{L})$. This was done by decomposing $S^3_{L}$ as a union of several \textit{types} of 3-manifolds with boundary, showing the result holds for each type, and using the gluing axiom.
\item In \seref{s:torus} we showed that $Z_{TV, \mathcal{C}}(\mathbb{T}^2) \cong Z_{RT, Z(\mathcal{C})}(\mathbb{T}^2)$ and that this isomorphism agrees with the action of the mapping class group $\Gamma(\mathbb{T}^2)$
\end{enumerate}
We will now connect these results. 
Let $K$ be a framed knot with framing $n \in \mathbb{Z}$ and $\mathcal{M}_{K}$ a 3-manifold with $K$ embedded. Let $\mathcal{M}'$ be the result of performing surgery in $\mathcal{M}$ along $K$.  We can write the gluing map $\varphi: \mathbb{T}^2 \to \partial (\M - \mathbb{T}^2)$ as $\varphi = T^n \circ S$.
\begin{lemma}\label{l:knotsurgery}
$Z_{TV, \mathcal{C}}(\mathcal{M}') = \frac{1}{\DD^2}Z_{TV,\mathcal{C}}(\mathcal{M}_{K})$.
\end{lemma}
\begin{proof}
Recall that $Z_{TV,\mathcal{C}}(\mathcal{M}_{K}) \equiv \displaystyle \sum_{i} d_{i}Z_{TV, \mathcal{C}}(\mathcal{M}_{K_i})$ , where $K_i$ is the knot $K$ colored by $i \in Irr(\mathcal{C})$ and
$\mathcal{M}_{K} = (\mathcal{M} - \textbf{T}^2) \sqcup_{Id} \textbf{T}^2$, where $\textbf{T}^2$ is a tubular neighborhood of $K$. Then
\begin{align*}
Z_{TV, \mathcal{C}}(\mathcal{M}_{K}) = (\displaystyle \sum d_i[i], Z(\mathcal{M}_{K} - \textbf{T}^2)) = (\displaystyle \sum d_i t^{n}_{i,i}[i], Z(\mathcal{M}_{K} - \textbf{T}^2))
\end{align*}
\begin{align*}
 =\DD^{2}(T^{n}_{*}S_{*}[\one],Z(\mathcal{M}_{K} - \textbf{T}^2) = \DD^{2}Z_{TV,\mathcal{C}}((\mathcal{M} - \textbf{T}^2) \sqcup_{T^{n} \circ S} \textbf{T}^2) = \DD^{2}Z_{TV, \mathcal{C}}(\mathcal{M}').
\end{align*}
\end{proof}
We can slightly generalize \leref{l:knotsurgery}. The proof is similar.
\begin{lemma}\label{l:linksurgery}
Let $\mathcal{M}_{L}$ be a 3-manifold with a framed link $L$ inside. Let $\mathcal{M}'_{L'}$  be the result of performing surgery on $\mathcal{M}_{L}$ along a single component of $L$. (Note that $|L'| = |L|-1$). Then $Z_{TV,\mathcal{C}}(\mathcal{M}'_{L'}) = \frac{1}{\DD^2}Z_{TV,\mathcal{C}}(\mathcal{M}_{L})$.
\end{lemma}
Note that if $L$ is a knot, the lemma reduces to \leref{l:knotsurgery}.
Now we state the main and final theorem of the paper, which relates the Reshetikhin-Turaev and Turaev-Viro invariants. The proof makes repeated use of \leref{l:linksurgery} and is very simple. Since we will be working with two categories, $\mathcal{C}$ and its Drinfeld Center $Z(\mathcal{C})$, we will replace all potentially ambiguous shorthand in what follows. For example, we will write Dim($\mathcal{C}$) instead of $\DD$.
\begin{theorem}
Let $\mathcal{M}$ be a closed, oriented 3-manifold with a colored link inside. Then $Z_{TV,\mathcal{C}}(\mathcal{M}) = Z_{RT, Z(\mathcal{C})}(\mathcal{M})$.
\end{theorem}
\begin{proof}
For simplicity, we can assume $\M$ has no embedded link. The general case is proved in the exact same fashion. We can obtain $\mathcal{M}$ from $S^3_{L}$ via surgery along a tubular neighborhood of $L$. Evidently, we can perform surgery along each component of $L$ individually. Using \leref{l:linksurgery} repeatedly, we obtain 
\begin{align*}
Z_{TV,\mathcal{C}}(\mathcal{M}) = \frac{1}{\Dim(\mathcal{C})^{2|L|}}Z_{TV, \mathcal{C}}(S^3_{L}) = \displaystyle \sum \frac{1}{\Dim(\mathcal{C})^{2|L|}}\prod d_{Y_i}Z_{TV,\mathcal{C}}(S^3;L_i;Y_i).
\end{align*}
 By Theorem 2.3, this equals
 \begin{align*}
  \displaystyle \sum \frac{1}{\Dim(\mathcal{C})^{(2|L|+2)}}\prod d_{Y_i}F(L_{i};Y_i) = \displaystyle \sum \frac{1}{\Dim(Z(\mathcal{C}))^{(|L|+1)}}\prod d_{Y_i}F(L_{i};Y_i) \\
  \\ \equiv Z_{RT,Z(\mathcal{C})}(\mathcal{M}).
\end{align*}
\end{proof}
We have shown that $Z_{TV, \C}$ and $Z_{RT, Z(\C)}$ give the same 3-manifold invariants. Using the gluing axiom and the fact that the theories agree on the n-puntured sphere, one can easily show that for $\Sigma_{g}$ a closed genus $g$ surface, $\Dim(Z_{TV,\C}(\Sigma_g)) = \Dim(Z_{RT, Z(\C)}(\Sigma_g))$ \footnote{In fact it is not hard to explicitly compute this this common dimension. See \ocite{turaev}}. Thus, the vector spaces associated to 2-manifolds are isomorphic. This is not enough, however, to show the TQFTs are isomorphic. We will have to construct a canonical isomorphism between the spaces and show these satisfy certain compatibility conditions. This will be done in an upcoming paper \ocite{mine2}.
\section{Appendix: Some Proofs}\label{s:appendix}
In this appendix, we include some of the longer computations described in the paper. \\
\subsection*{Proof of \leref{l:braid}} 
\begin{figure}[ht]
\begin{center}
\figscale{.5}{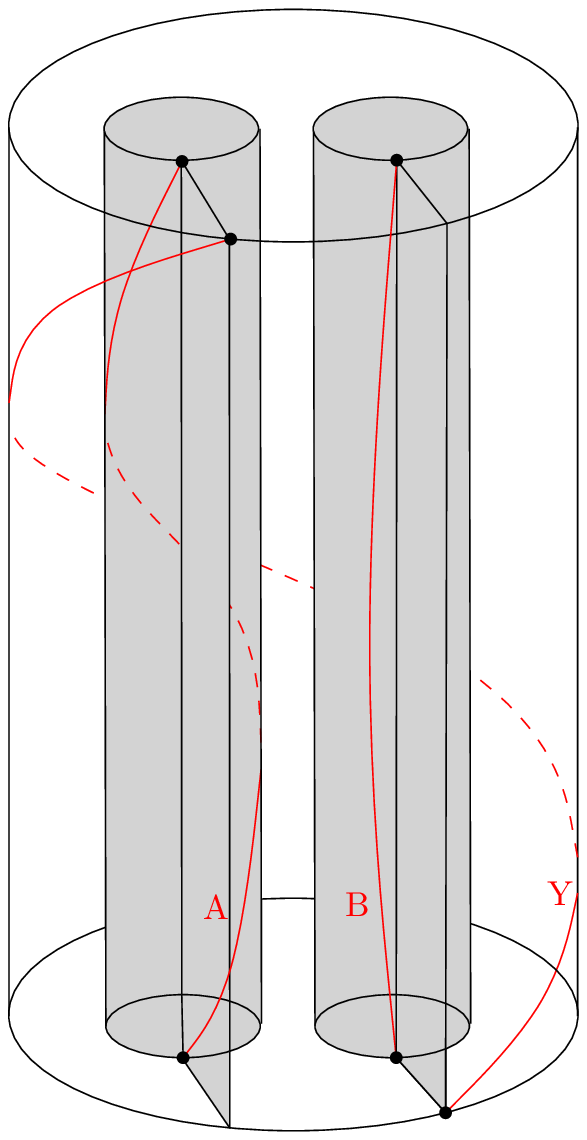}
\end{center}
\end{figure}
Decompose $\N$ into a combinatorial 3-manifold.
The tubes labeled $A$ and $B$ are shaded gray and the tube labeled $Y$ lies on the outside of the 3-cell. The decomposition has 4 vertices and 16 edges, of which 4 are internal and 12 lie on the boundary. The decomposition has four 3-cells, of which 3 are labeled tubes. Orienting and coloring edges, we get 4 graphs, one for each 3-cell. Note that in the following diagrams we often replace coupons with vertices and omit labeling vertices by the appropriate vectors. As always, we label vertices corresponding to dual Hom-spaces with dual basis vectors and sum over all these 'paired' bases (Recall that these correspond to internal 2-cells). We also label the bottom and top vertices by $\varphi$ and $\varphi'$ respectively.

\begin{figure}
\figscale{.9}{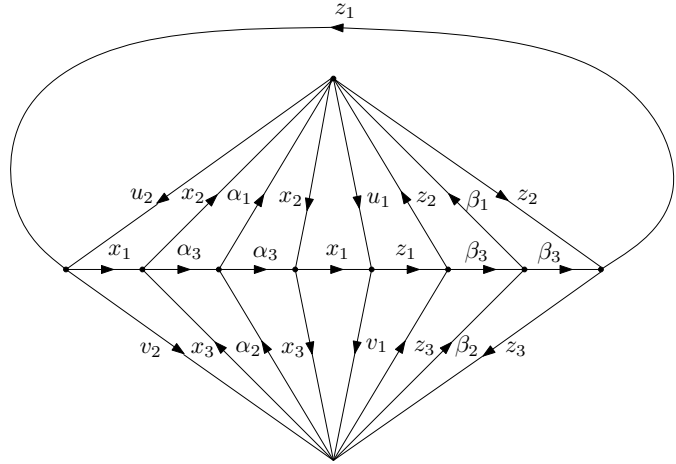}
\caption{Big 3-cell}
\end{figure}

\begin{figure}
\figscale{.6}{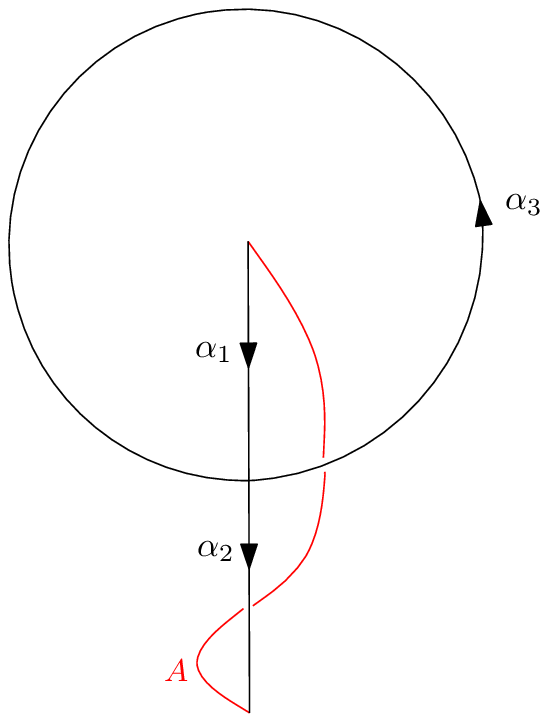} \figscale{.6}{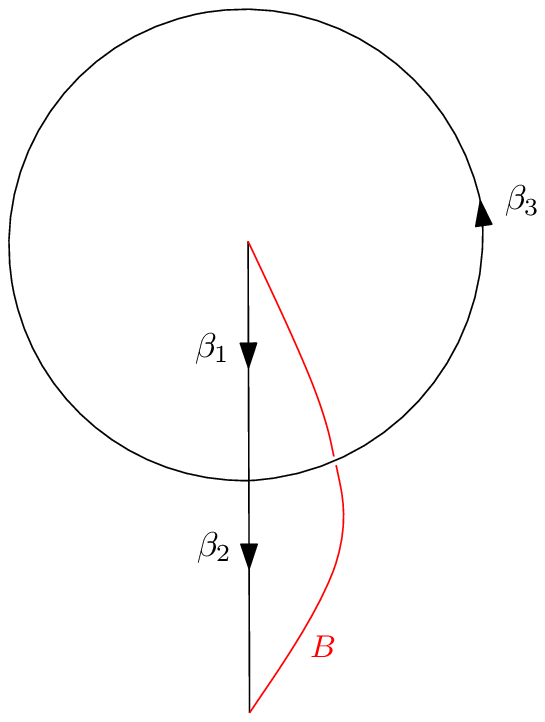} \figscale{.6}{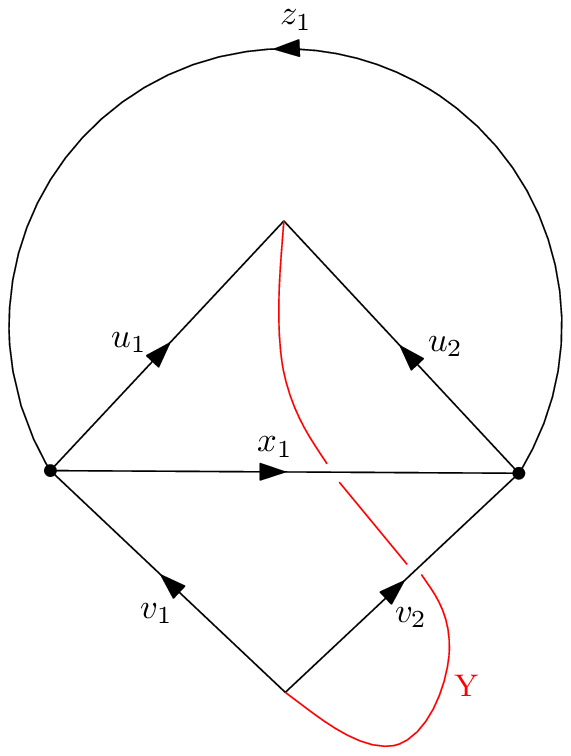}
\caption{Dual graphs of the 3 tube cells}
\end{figure}

The state sum is therefore: 
\begin{center}
$Z_{TV}(\N)=\sum D K_{3}K_{Y}K_{A}K_{B}$ \\
\end{center}
where $K_{3}, K_{Y}, K_{A}, K_{B}$ denote the evaluations of each of the labeled graphs picture above and $D$ is the following unsightly term:
\begin{center}
$D = \DD^{-8}\sum d_{x_1}d_{z_1}d_{\alpha_3}d_{\beta_3}d_{\alpha_1}^{\frac{1}{2}}d_{\alpha_2}^{\frac{1}{2}}d_{\beta_1}^{\frac{1}{2}}d_{\beta_2}^{\frac{1}{2}}d_{x_2}^{\frac{1}{2}}d_{x_3}^{\frac{1}{2}}d_{z_2}^{\frac{1}{2}}d_{z_3}^{\frac{1}{2}}d_{u_1}^{\frac{1}{2}}d_{u_2}^{\frac{1}{2}}d_{v_1}^{\frac{1}{2}}d_{v_2}^{\frac{1}{2}}$
\end{center}
Recall that each vertex is labeled by a morphism in the corresponding Hom-space. Pairing dual morphisms, we can glue in the 3 tubes. 
\begin{center}
$Z_{TV}(\N)= \sum  \DD^{-8} d_{x_1}d_{z_1}d_{\alpha_3}d_{\beta_3}d_{x_2}^{\frac{1}{2}}d_{x_3}^{\frac{1}{2}}d_{z_2}^{\frac{1}{2}}d_{z_3}^{\frac{1}{2}}d_{u_1}^{\frac{1}{2}}d_{u_2}^{\frac{1}{2}}d_{v_1}^{\frac{1}{2}}d_{v_2}^{\frac{1}{2}} \figscale{.6}{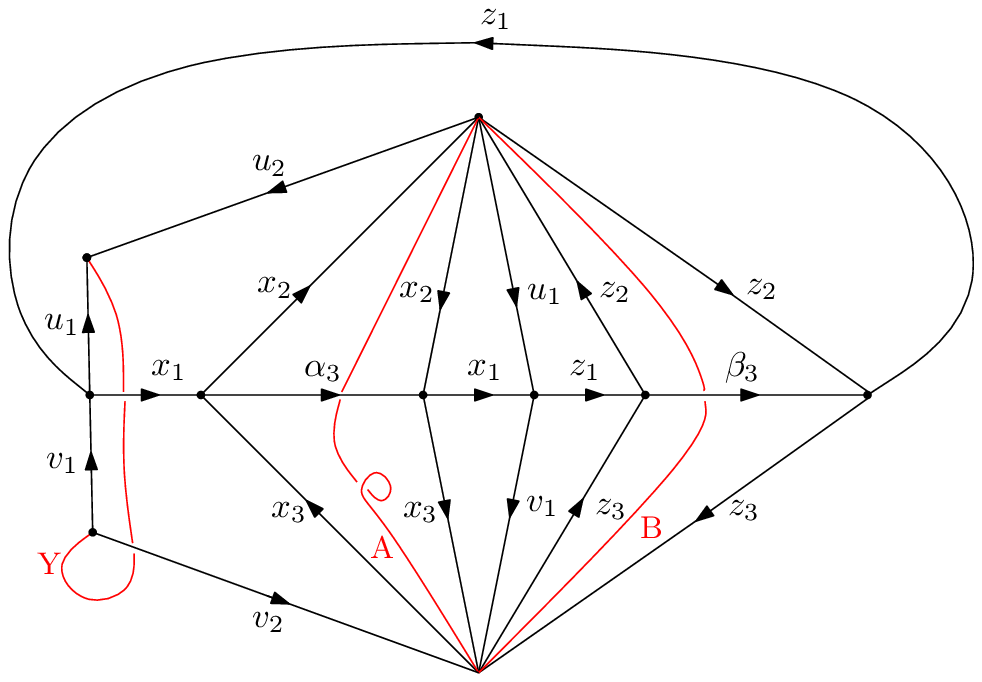}$
\end{center}
Here we have used the pairing of dual graphs and semisimplicity (\leref{l:pairing}).
Using \leref{l:pairing} three more times, we get
\begin{center}
$Z_{TV}(\N)= \sum \DD^{-8} d_{z_1}d_{x_2}^{\frac{1}{2}}d_{x_3}^{\frac{1}{2}}d_{z_2}^{\frac{1}{2}}d_{z_3}^{\frac{1}{2}}d_{u_1}^{\frac{1}{2}}d_{u_2}^{\frac{1}{2}}d_{v_1}^{\frac{1}{2}}d_{v_2}^{\frac{1}{2}} \figscale{.6}{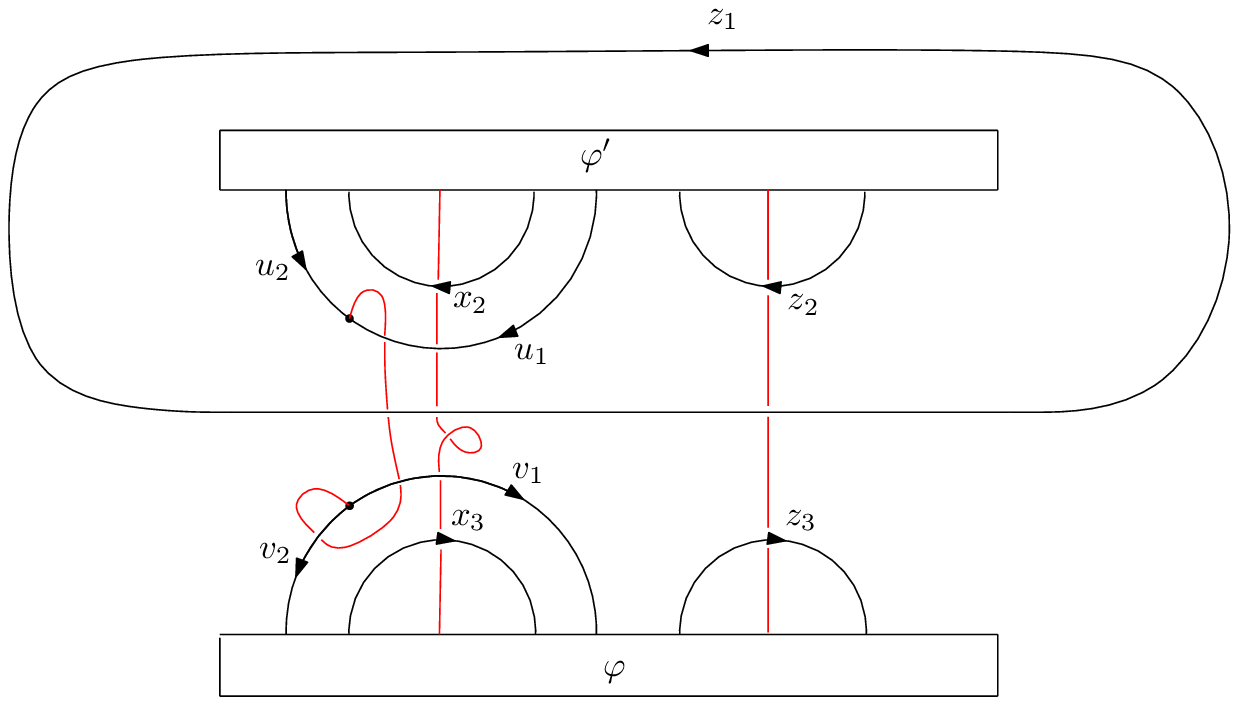}$
\end{center}
If we pair some more and cancel opposite twists, this equals
\begin{center}
$ \sum \DD^{-8} d_{z_1}d_{x_2}^{\frac{1}{2}}d_{x_3}^{\frac{1}{2}}d_{z_2}^{\frac{1}{2}}d_{z_3}^{\frac{1}{2}}d_{u_2}^{\frac{1}{2}}d_{v_1}^{\frac{1}{2}} \figscale{.6}{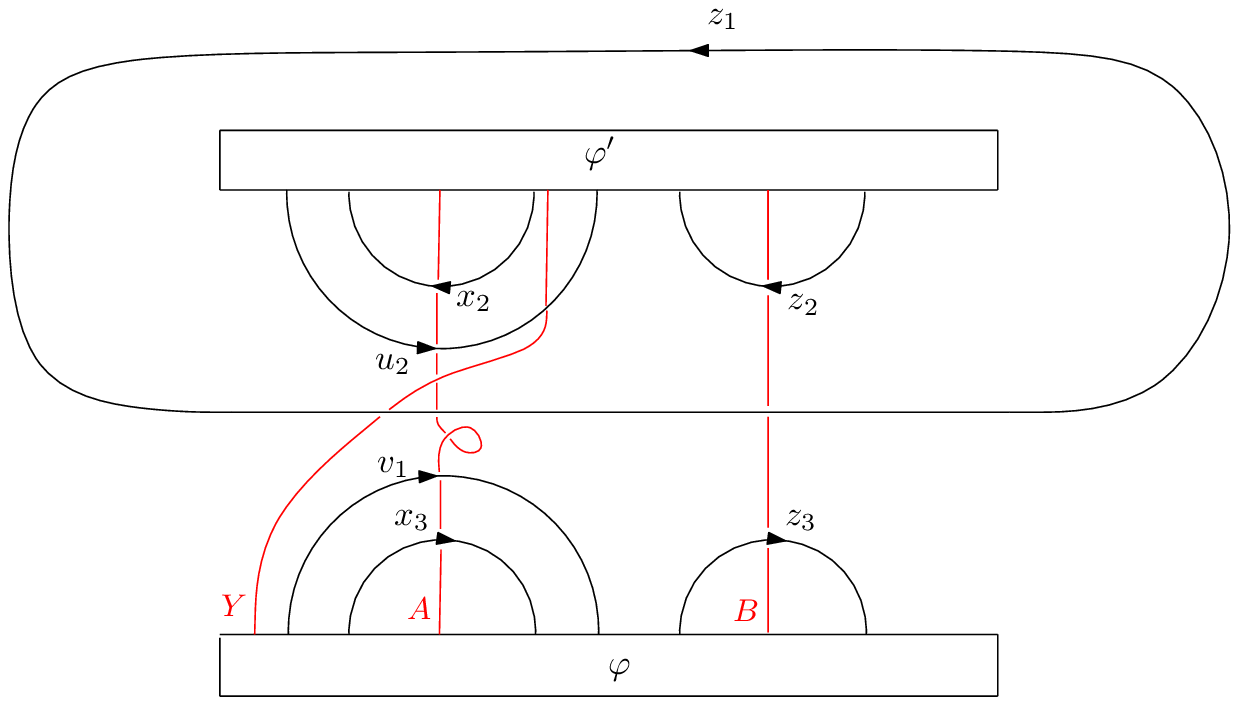}$
\end{center}

It follows from Lemma 2.2 and Example 8.6 of \ocite{mine} that the loop labeled $z_1$ projects the top morphism onto $Hom_{Z(\C)}(1, Y\otimes A\otimes B)$.  
After projection, the diagram may be depicted as below where $P$ is the projector from \leref{l:projector} and the diagram on the left follows from the fact that for a projector, ($P\varphi, \varphi'$) = ($P\varphi, P\varphi'$). The final isomorphism may be easily verified by the reader.
\begin{align*}
\figscale{.7}{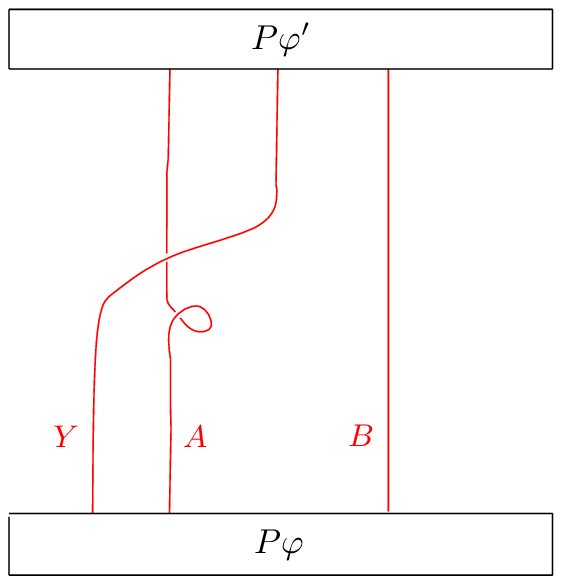} \hspace{.2cm}\cong \hspace{.3cm}  \figscale{.7}{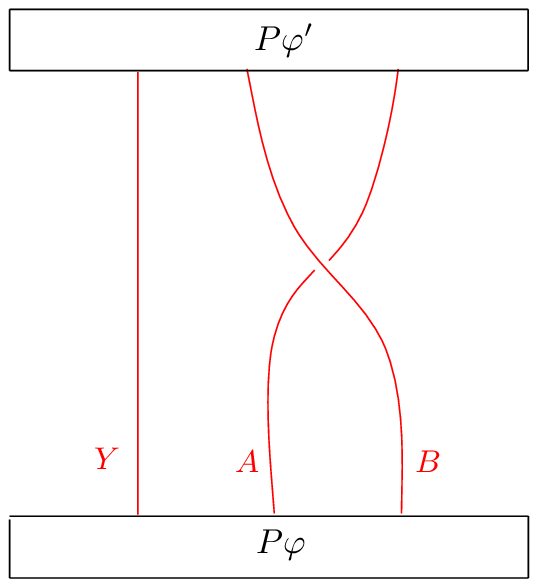}  = \big{(}P\varphi',(1_{Y}\otimes \sigma_{AB})P\varphi \big{)}
\end{align*}

\subsection*{Proof of \leref{l:cap}}
We choose the following combinatorial structure on $\mathcal{M}$:
\begin{itemize}
\item Three 3-cells, one of which is an open embedded tube, with longitude labeled by Y$\in Irr(Z(\mathcal{C}))$.
\item 8 vertices, all of which lie on the boundary 
\item 15 edges, three of which are internal.
\end{itemize}
\begin{figure}[ht]
\figscale{.5}{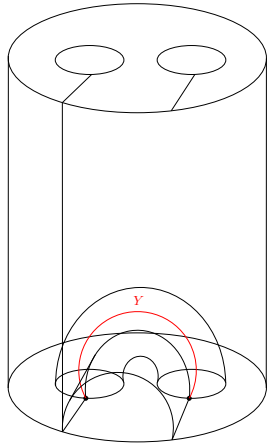}
\caption{Decomposition of $\mathcal{M}$}
\end{figure}
For each 3-cell, we get a graph.
\begin{figure}[ht]
\begin{center}
\figscale{.5}{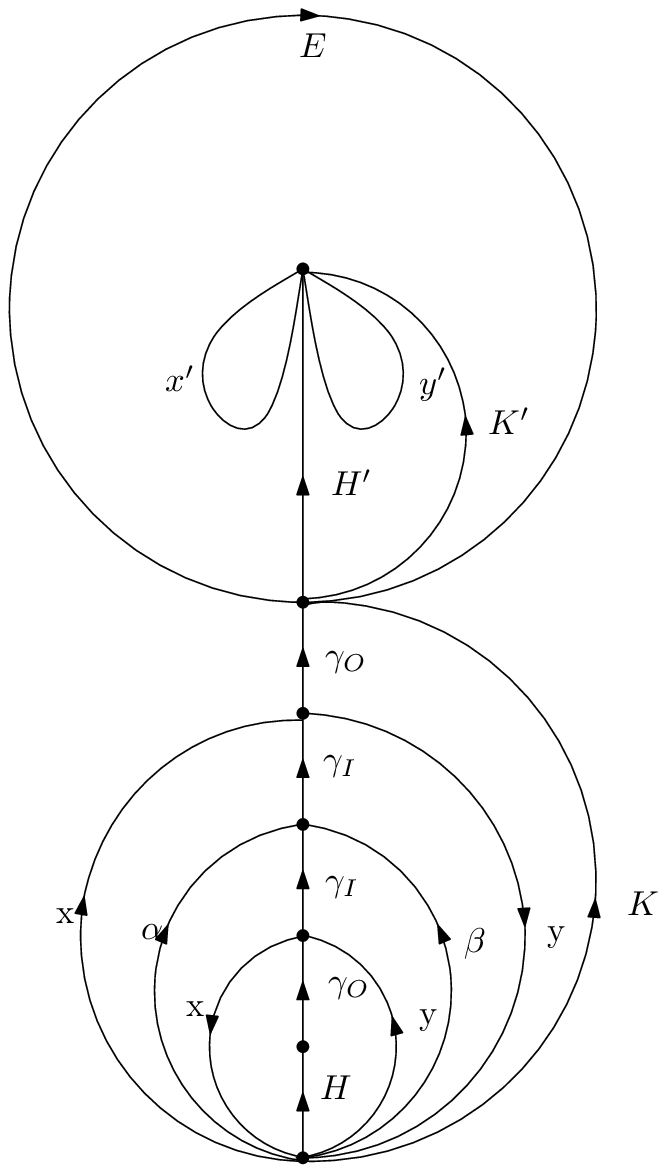} \hspace{.1cm} \figscale{.5}{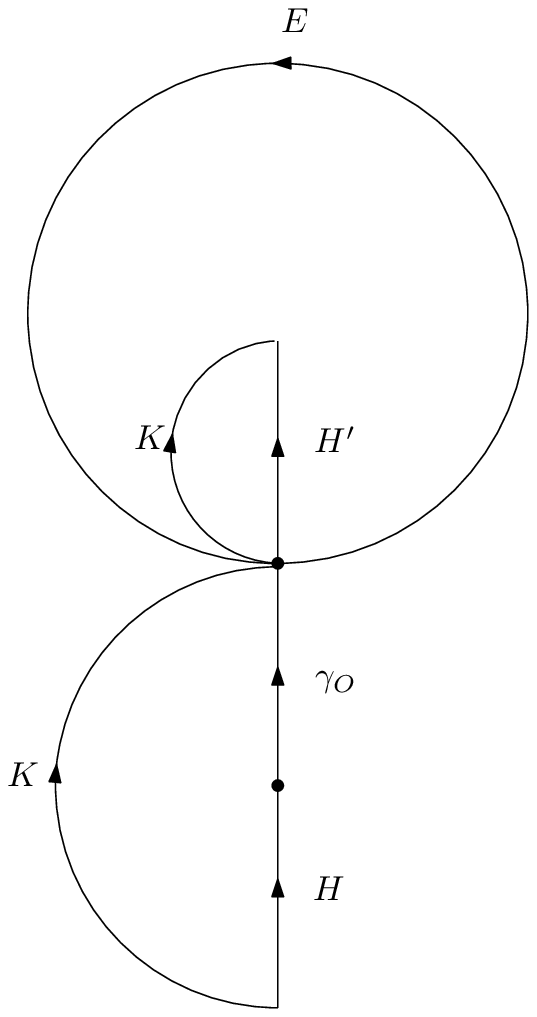} \hspace{.1cm} \figscale{.5}{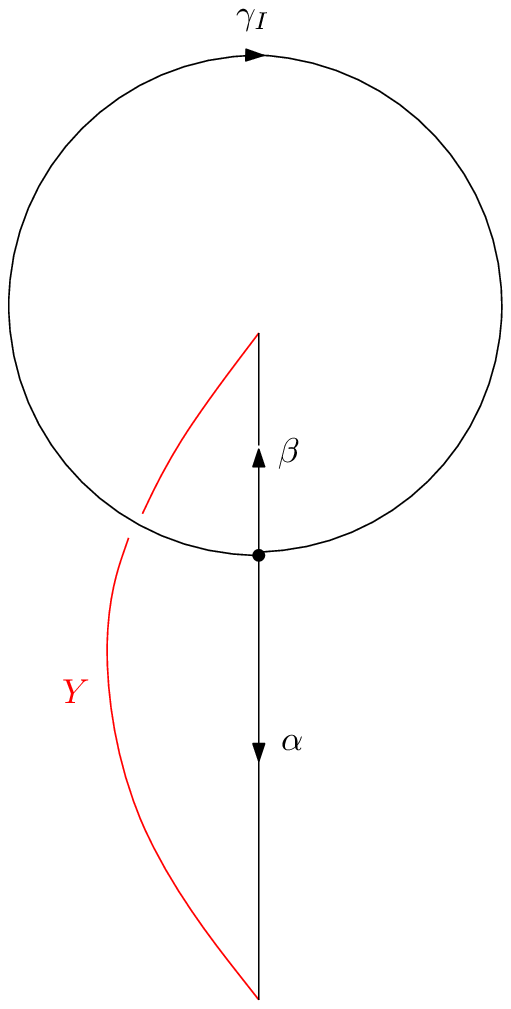}
\caption{
The graphs corresponding to the three 3-cells from Figure 6. From left to right, these correspond to the main 3-cell, the outer 3-cell and the embedded tube cell}
\end{center}
\end{figure}
 Gluing these graphs together using \leref{l:pairing} yields the following graph:
\begin{align*}
\figscale{.5}{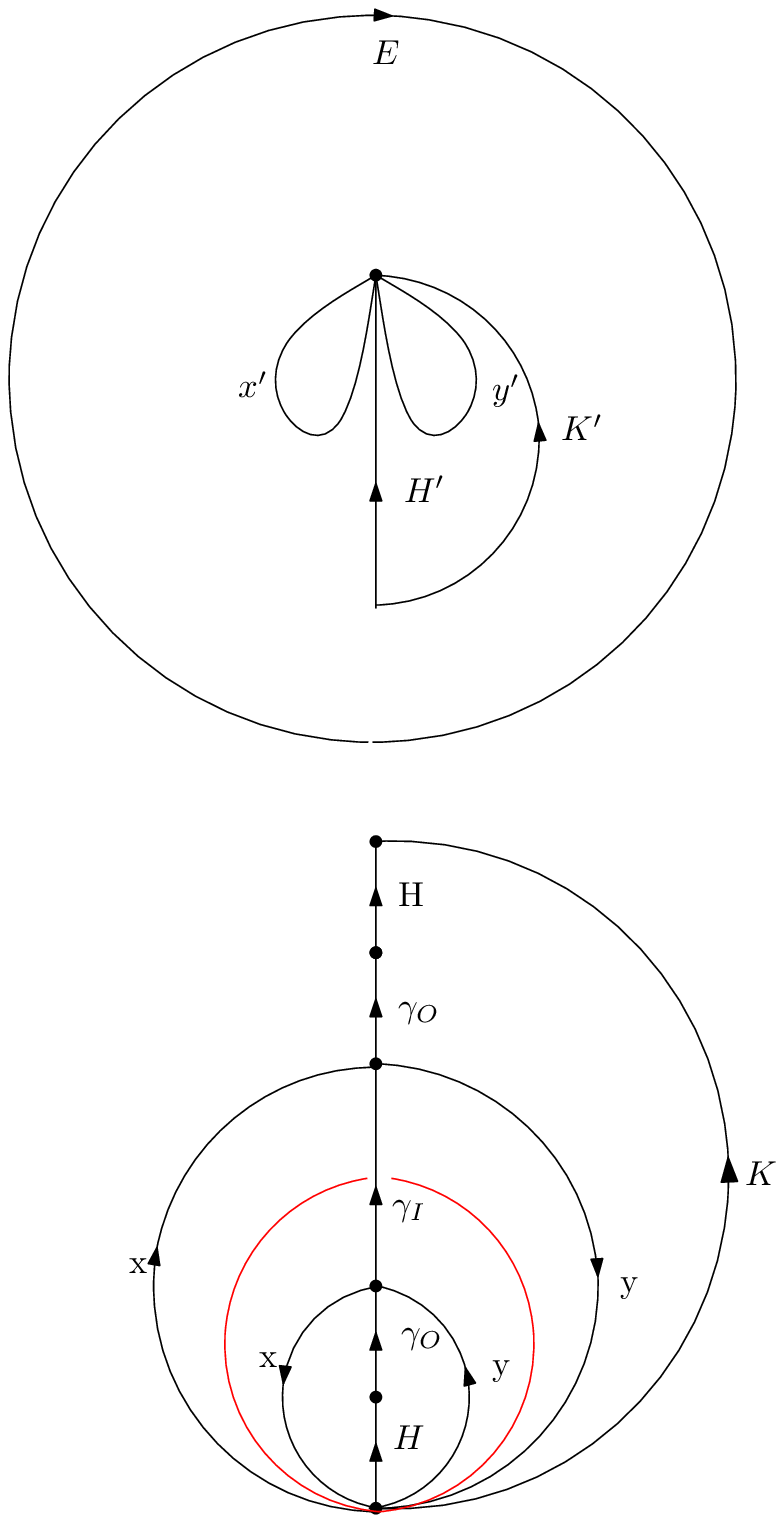}
\end{align*}

Simplifying the graph using \leref{l:pairing}, the state-sum formula becomes
\begin{align*}
Z_{TV,\mathcal{C}}(M) = \sum \DD^{-8}d_{E}(d_{x}d_{x'}d_{y}d_{y'}d_{K}d_{K'})^{\frac{1}{2}}\figscale{.7}{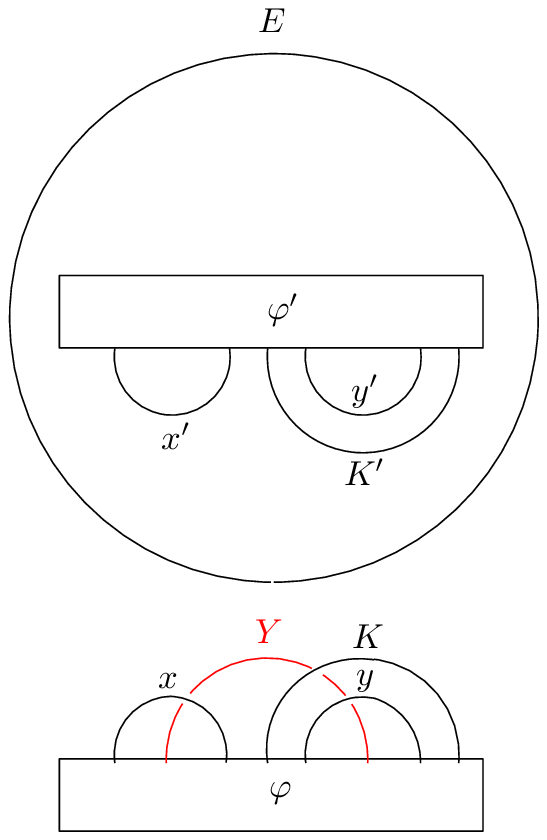}
\end{align*}

Using the same arguments as the previous calculation, one can see that this gives ($P\varphi ', ev_Y( P\varphi$)), where $ev_Y: Y^* \otimes Y \to \one$ is the evaluation map in Z($\C$).
\subsection{Computation of $Z_{TV}(\mathbb{T}^2 \times I)$ from \leref{l:torusspace}}

Let us take the following polytope decomposition of $\mathbb{T}^2$ consisting of one vertex, two edges and one face. 
\begin{figure}[ht]
\figscale{.5}{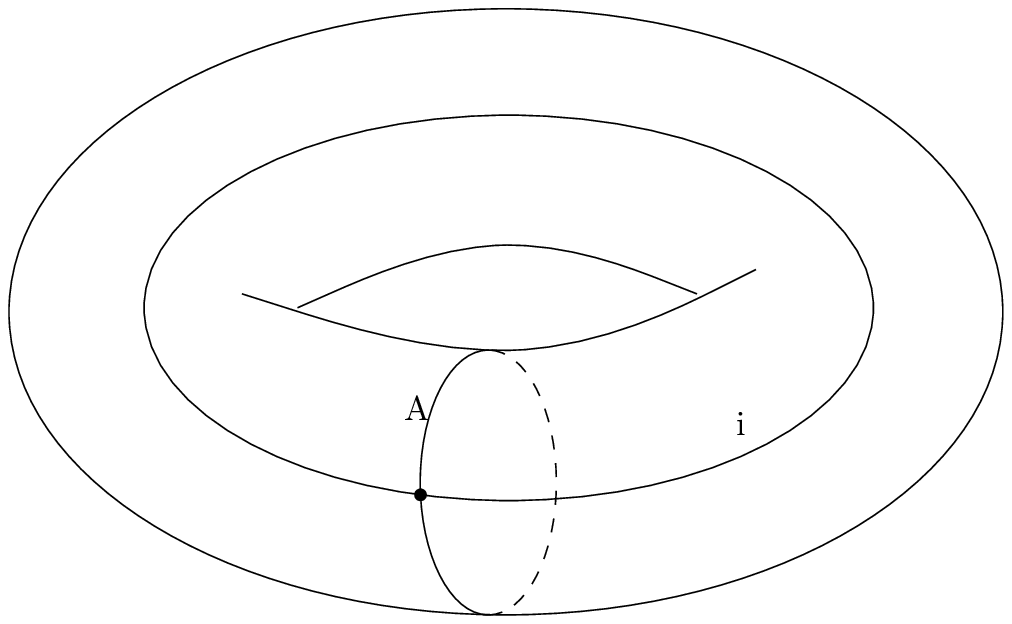}
\end{figure}
Then  \begin{center}
$H(\mathbb{T}^2)= \displaystyle\bigoplus_{A,X \in Irr{\mathcal{C}}}{\<A,X,A^{*},X^{*}\>}$. \end{center}. We will often make use of the following isomorphisms:
\begin{equation}\label{e:isom}
H(\mathbb{T}^2) \cong \displaystyle\bigoplus_{Z,A}\<Z,A\> \otimes \<Z^{*},A^{*}\> \cong \displaystyle \bigoplus_{Z}\<Z,Z^{*}\>_{\mathcal{C}}
\end{equation} 

 where the first isomorphism is given by the map $G$ from \leref{l:gluing_isom2}, and the second is given by a direct sum of  composition maps (\leref{l:composition2}). Choosing bases $\{\varphi_{Z,A,i}\}$ in $\<Z,A\>$ and $\{\psi_{Z,A,j}\}$ in $\<Z^*,A^*\>$, we can write a basis in $H(\mathbb{T}^2)$ as
 \begin{align*}
\eta_{Z,A,i,j} = \displaystyle \sum_{X \in Irr(\mathcal{C})}\frac{\sqrt{d_X}\sqrt{d_Z}}{\mathcal{D}}\figscale{.5}{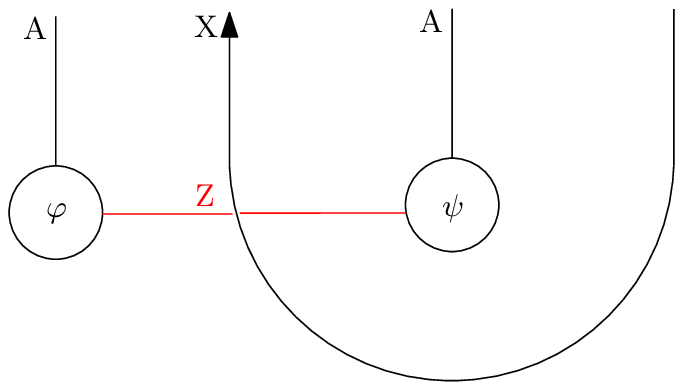}
\end{align*} 
Composing the map $G$ with the direct sum of  composition maps, we see that $H(\mathbb{T}^2) \cong \displaystyle \bigoplus_{Z}\<Z,Z^*\>_{\mathcal{C}}$. We now show that $Z_{TV,\mathcal{C}}(\mathbb{T}^2 \times I)$ computes the projection onto $\displaystyle \bigoplus_{Z \in Irr(Z(\mathcal{C}))}\<Z,Z^*\>_{Z(\mathcal{C})}$.

Consider the decomposition of $\mathbb{T}^2 \times I$ shown below. 
\begin{figure}[ht]
\figscale{.5}{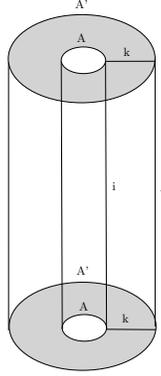}
\caption{Decomposition of $\mathbb{T}^2 \times I$}
\end{figure}.
The top and bottom of the figure are shaded to emphasize that they are to be identified.
This decomposition consists of 2 vertices both on the boundary, 5 edges 4 of which lie on the boundary, and a single 3-cell.\\

$Z_{TV}(\mathbb{T}^2 \times I) = \displaystyle \sum \frac{1}{\mathcal{D}^2}(d_{A}d_{A'}d_{i}d_{j})^{\frac{1}{2}}d_{k}\figscale{.5}{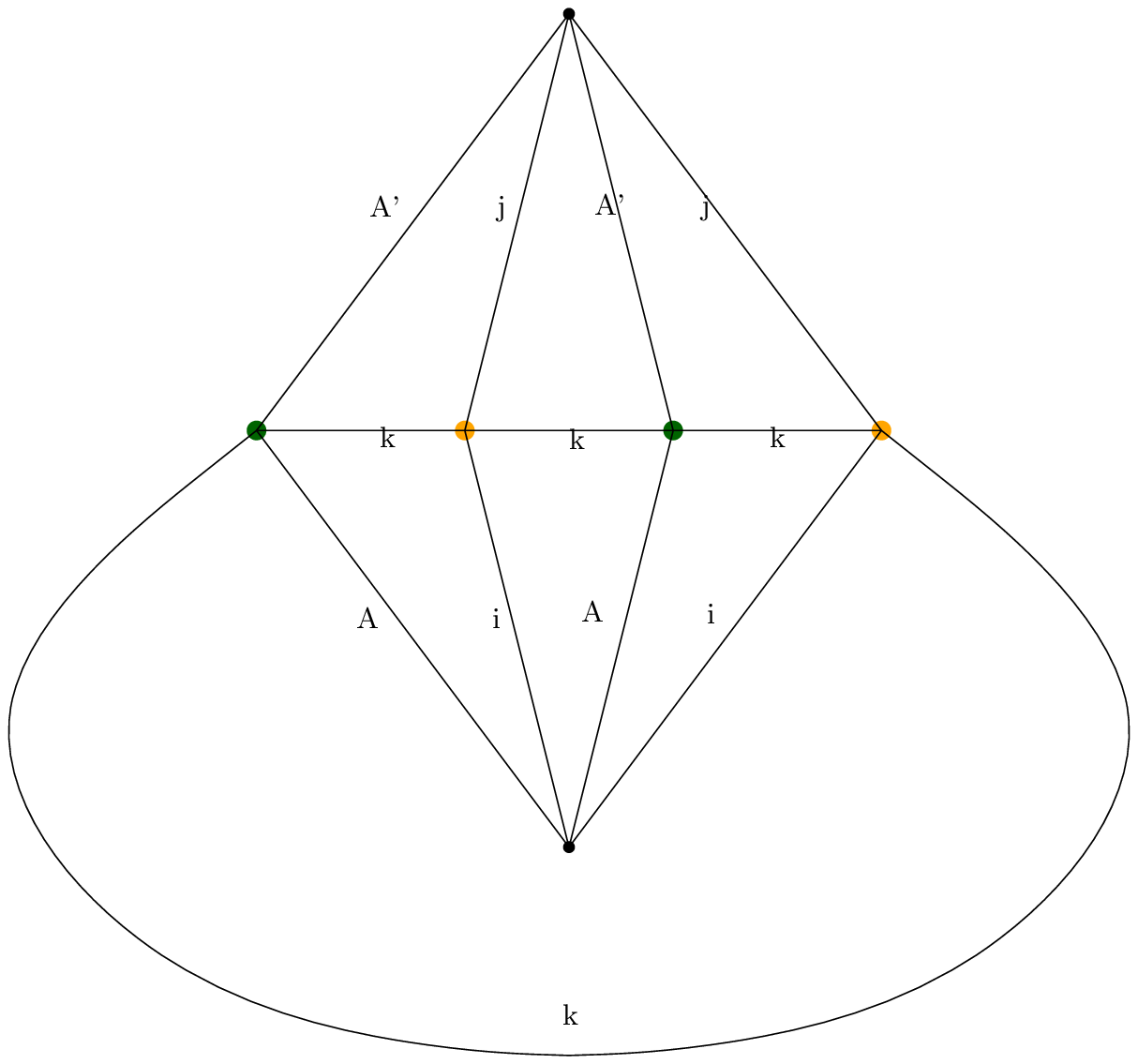}$
where we have labelled vertices dual to one another by the same color.
This gives a map $\Phi: \displaystyle \bigoplus_{A,i} \<A,X_{i},A^*,X_{i}^*\> \longrightarrow \bigoplus_{A',j} \<A',X_{j},A'^*,X_{j}^*\>$. Composing on both sides by $G$ yields a map $G^{-1} \Phi G: \displaystyle \bigoplus_{Z,A}\<Z,A\> \otimes \<Z^*,A^*\> \longrightarrow \displaystyle \bigoplus_{W,A'}\<W,A'\> \otimes \<W^*,A'^*\>$.
\begin{align*}
\displaystyle \sum \frac{1}{\mathcal{D}^4} (d_{A}d_{A'})^{\frac{1}{2}}d_{i}d_{j}d_{k}\figscale{.4}{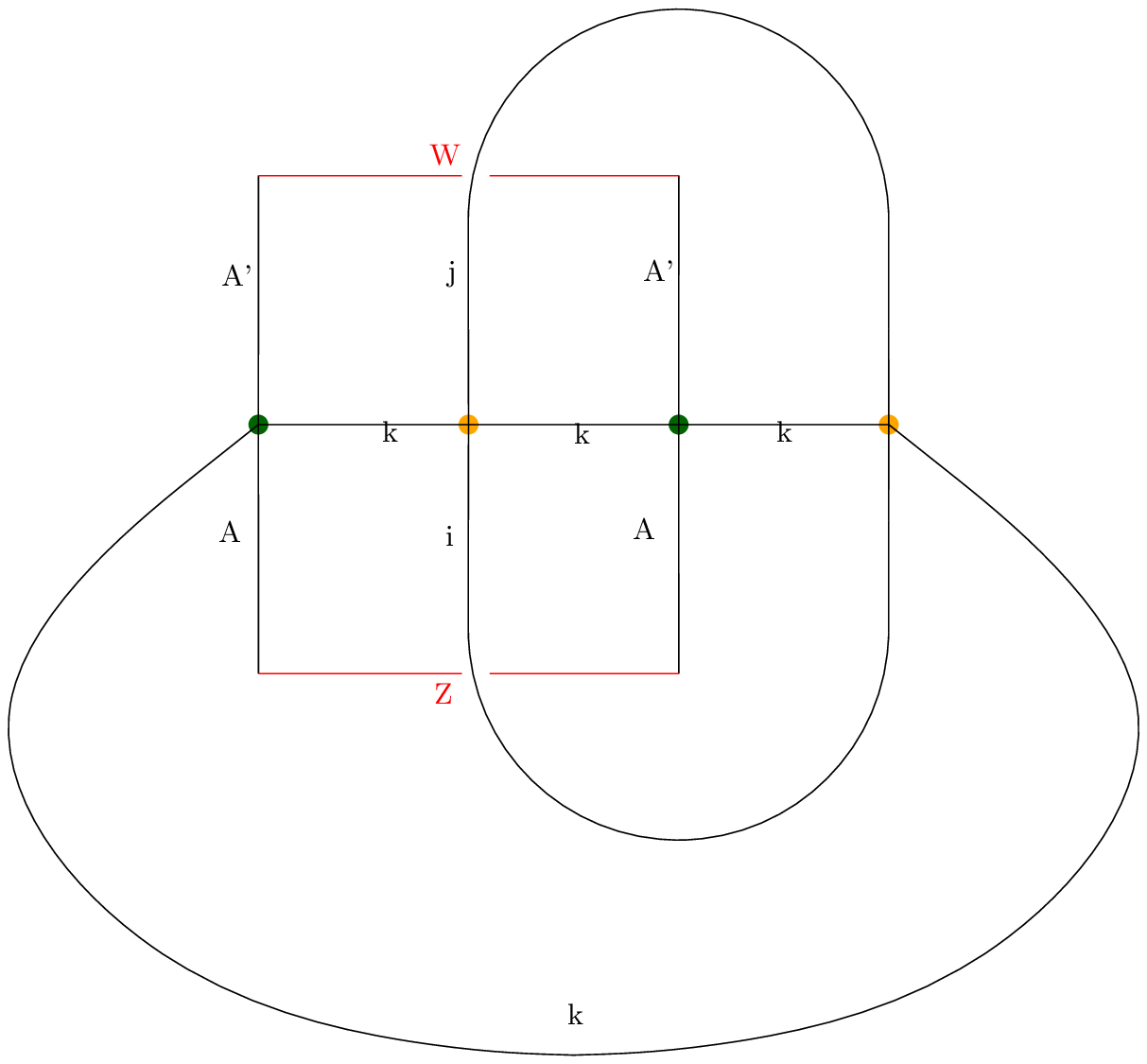}
\end{align*}
Pairing the two orange vertices yields
\begin{align*}
\displaystyle \bigoplus_{Z,W} \sum \frac{1}{\mathcal{D}^4} (d_{A}d_{A'})^{\frac{1}{2}}d_{i}d_{k}\figscale{.4}{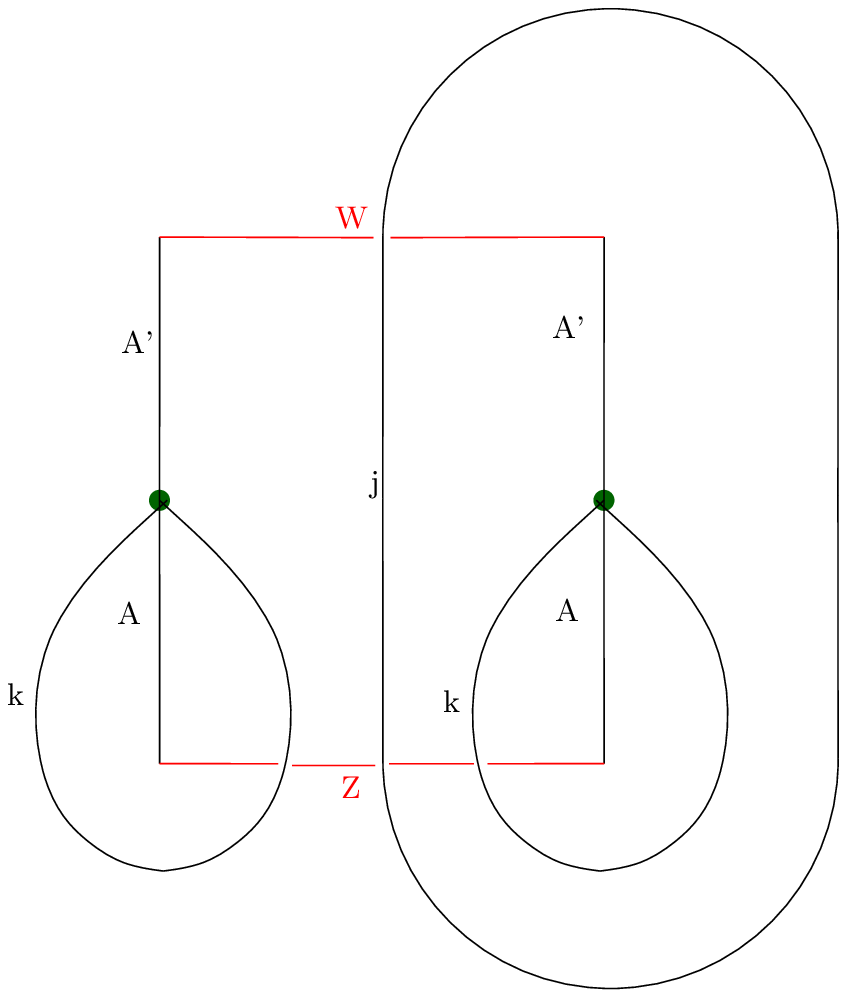}
\end{align*}
By Schur's Lemma, this diagram evaluates to 0 unless $W = Z^*$. If $W = Z^*$, this equals (\leref{l:projector})
 \begin{align*}
\displaystyle \bigoplus_{Z} \sum \frac{1}{\mathcal{D}^2} (d_{A}d_{A'})^{\frac{1}{2}}d_{k}\figscale{.5}{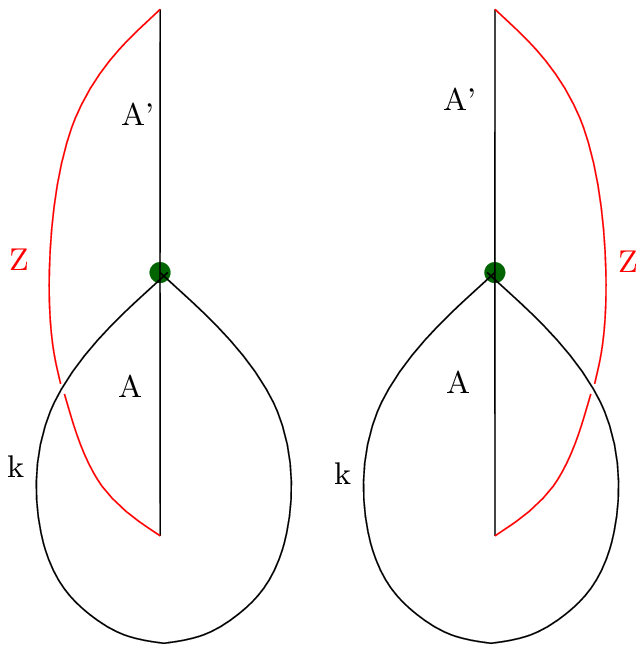} = \displaystyle \bigoplus_{Z} \sum \frac{1}{\mathcal{D}^2}d_{k}\figscale{.5}{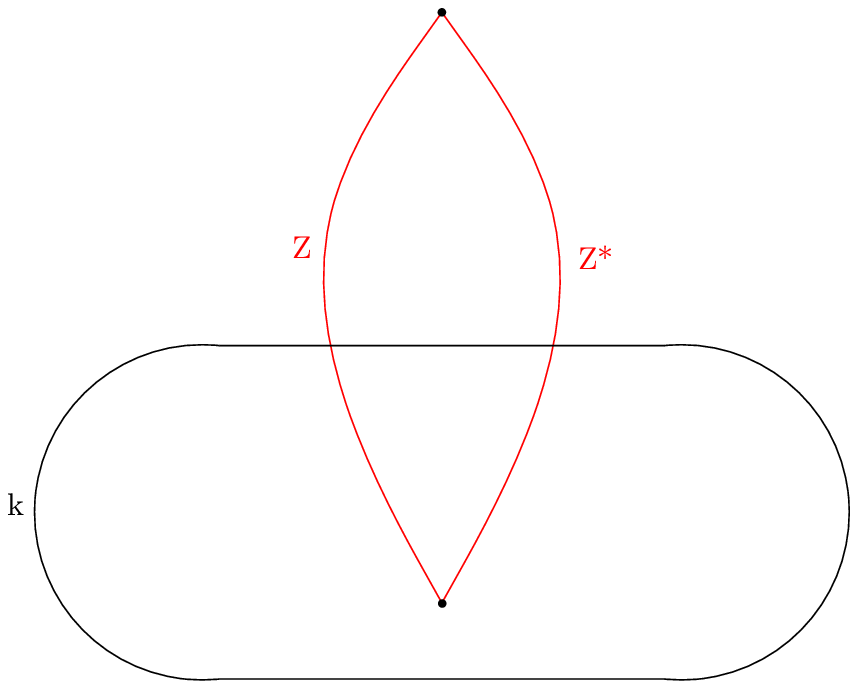}
\end{align*}
where the equality follows from pairing the dual vertices and using the rescaled composition map twice. By \leref{l:projector} this is a projector onto
\begin{align*}
 \displaystyle \bigoplus_{Z \in Irr(Z(\mathcal{C}))}Hom_{Z(\mathcal{C})}(\one, Z \otimes Z^{*}).  
\end{align*}
This computation also shows that the projection is compatible with \ref{e:isom}.
\subsection*{Proof of \leref{l:innerproduct}}

$\textbf{T}^2_{Z} \sqcup_{U} \textbf{T}^2_{W} = S^2 \times S^1$ with two unlinked embedded tubes labelled by $Z$ and $W$. Choose a decomposition of $S^2 \times S^1$ as pictured in  Figure 9.
\begin{figure}[ht] 
\figscale{.5}{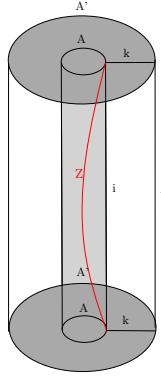}
\caption{Decomposition of $S^2 \times S^1$ with two closed embedded tubes. The tube labelled $W$ is to be glued on the outside of the pictured cylinder and the top and bottom of the picture are identified.} 
\end{figure}
This decomposition has two vertices, five edges and three 3-cells, two of which are embedded tubes. The computation of the state sum is similar to that in \leref{l:torusspace}. We get
\begin{align*}
Z(\textbf{T}^2_{Z} \sqcup_{U} \textbf{T}^2_{W}) = \displaystyle \sum \frac{1}{\mathcal{D}^4} d_{A}d_{A'}d_{i}d_{k} \figscale{.5}{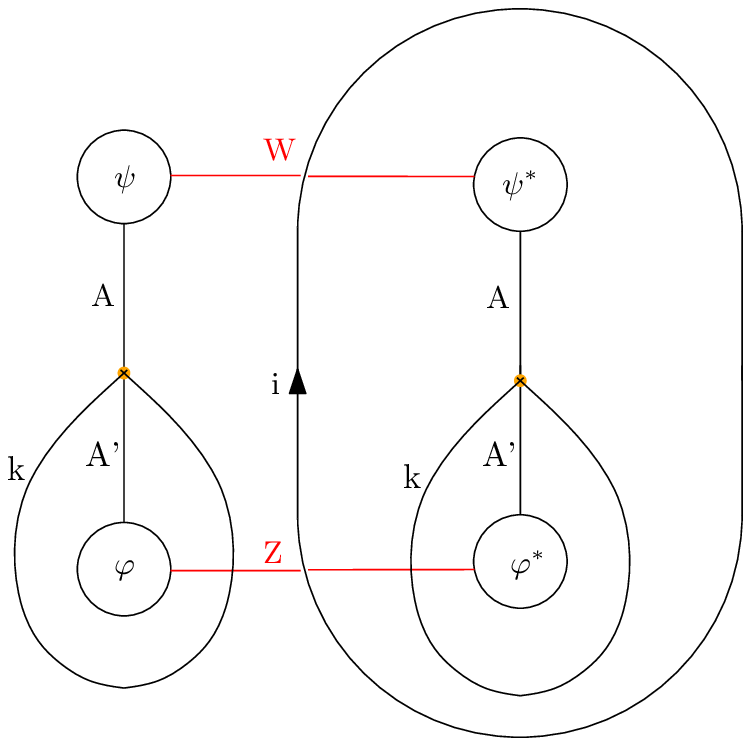}
\end{align*}
where the orange vertices are labelled by dual vectors. By \leref{l:projector} this equals 
\begin{align*}
\delta_{Z,W}\displaystyle \sum_{A,A',k} \frac{1}{\mathcal{D}^2} \frac{d_{A}d_{A'}d_{k}}{d_{Z}} \figscale{.5}{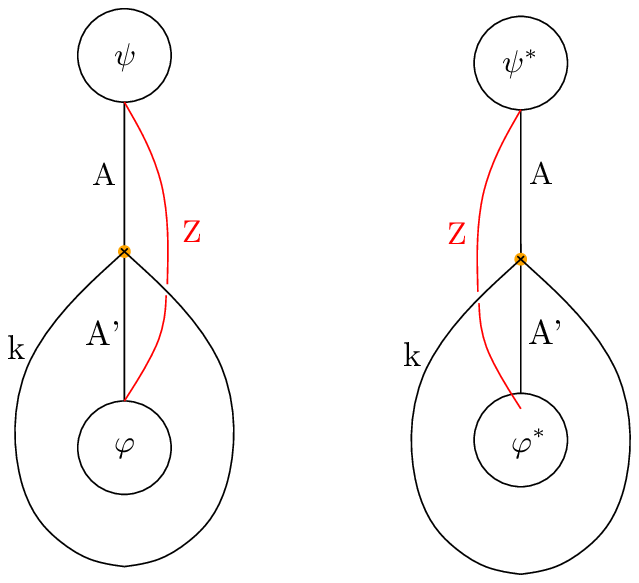}.
\end{align*}
Finally, pairing dual vertices using \leref{l:pairing}, we get
\begin{align*}
\delta_{Z,W}\displaystyle \sum_{A,A',k} \frac{1}{\mathcal{D}^2} \frac{d_{A}d_{A'}d_{k}}{d_{Z}} \figscale{.5}{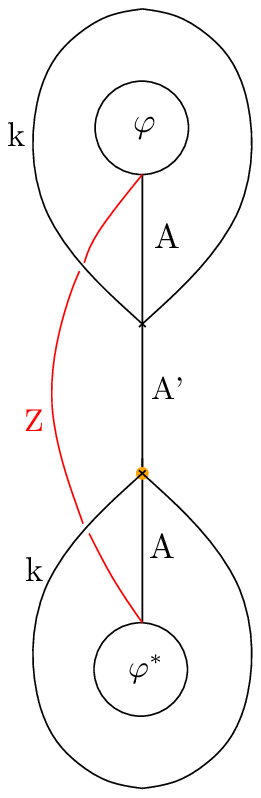} = \delta_{Z,W}\displaystyle \sum_{A,k} \frac{1}{\mathcal{D}^2} \frac{d_{A}d_{k}}{d_{Z}} \figscale{.5}{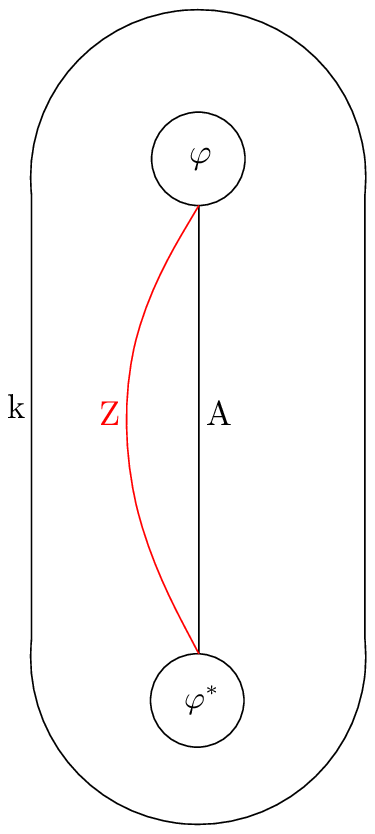} = \delta_{Z,W} \displaystyle \sum_{k} \frac{1}{\mathcal{D}^2} \frac{d_{k}}{d_{Z}} \figscale{.5}{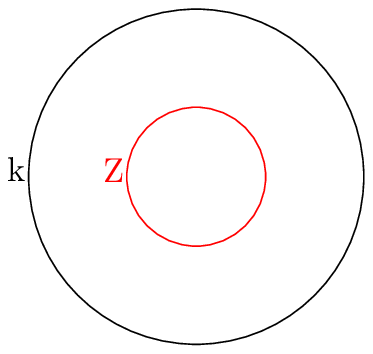} = \delta_{Z,W}
\end{align*}

\end{document}